\documentclass[11pt,a4paper]{article}

\usepackage[utf8]{inputenc}
\usepackage[T1]{fontenc}
\usepackage{lmodern}
\usepackage{amsmath,amssymb,amsthm}
\usepackage{mathtools}
\usepackage{geometry}
\usepackage{graphicx}
\usepackage{booktabs}
\usepackage{caption}
\usepackage{placeins}
\usepackage{hyperref}
\usepackage{enumitem}
\usepackage{cleveref}
\usepackage{xcolor}
\usepackage{csquotes}
\usepackage{pgfplots}
\pgfplotsset{compat=1.18}
\usepackage{listings}
\usepackage[backend=biber,
style=numeric,
sorting=nyt,
giveninits=true,
doi=false,
isbn=false,
url=false,
eprint=false,
maxbibnames=99]{biblatex}

\usepackage[framemethod=tikz]{mdframed}
\mdfsetup{
  backgroundcolor=gray!5,
  linecolor=gray!50,
  linewidth=0.8pt,
  roundcorner=4pt,
  skipabove=1em,
  skipbelow=1em,
  innerleftmargin=10pt,
  innerrightmargin=10pt,
  innertopmargin=8pt,
  innerbottommargin=8pt
}

\usepackage{tikz}
\usetikzlibrary{arrows.meta,positioning,fit,calc, backgrounds}

\pgfdeclarelayer{background}
\pgfsetlayers{background,main}

\newcommand{\tikzmark}[1]{%
  \tikz[remember picture,overlay]\coordinate (#1);%
}

\newcounter{todo}
\theoremstyle{plain}
\newtheorem{theorem}{Theorem}
\newtheorem{lemma}{Lemma}
\newtheorem{proposition}{Proposition}
\newtheorem{corollary}{Corollary}

\theoremstyle{definition}
\newtheorem{definition}{Definition}
\newtheorem{example}{Example}

\theoremstyle{remark}
\newtheorem{remark}{Remark}

\theoremstyle{plain}
\newtheorem{conjecture}{Conjecture}
\theoremstyle{definition}
\newtheorem{openquestion}{Open question}
\newenvironment{boxedconjecture}{\begin{mdframed}\begin{conjecture}}{\end{conjecture}\end{mdframed}}
\newenvironment{boxedquestion}{\begin{mdframed}\begin{openquestion}}{\end{openquestion}\end{mdframed}}

\DeclareMathOperator{\N}{\mathbb{N}}
\DeclareMathOperator{\R}{\mathbb{R}}

\title{From Global Optimization to Exact Coordinates:\\
Heilbronn's Triangle Problem on the Unit Square\\
Using Mixed-Integer Nonlinear Optimization}
\author{Nathan Sudermann-Merx\thanks{Institute for Computer Science, DHBW Mannheim. Email: \texttt{nathan.sudermann-merx@dhbw.de}}}
\date{October 2026}

\begin{document}

\maketitle

\begin{abstract}
  We develop a mixed-integer nonlinear programming (MINLP) approach for the Heilbronn triangle problem on the unit square.
  The formulation exploits the boundary structure of optimal configurations, which we establish rigorously, to fix coordinates and orientation variables a~priori.
  With this formulation, a general-purpose global solver solves the case $n=9$ to global optimality within its numerical tolerances in about eight minutes on a laptop and without any precomputation, compared with about 23 hours of total computation for the best previous approach.
  The same model also solves $n=10$ for the first time, in under seven hours, and confirms the configuration of Comellas and Yebra (2002).
  The numerical solutions then guide an exact symbolic reconstruction, and the minimum triangle area of every reconstructed configuration is verified in exact arithmetic.
This yields exact coordinates for all $n\le 10$, including the configurations of Comellas and Yebra (2002) for $n=9$ and $n=10$, together with rigorous lower bounds.
  Finally, we derive structural properties of optimal configurations, formulate a conjecture on the number of critical triangles, and pose an open question.
  All code and data are publicly available.
\end{abstract}

\section{Introduction}

\subsection{The problem}

The Heilbronn triangle problem arises in discrete geometry and asks:

\begin{mdframed}
  \centering
  \emph{How should $n$ points be placed in the unit square so that the smallest triangle determined by any three of them is as large as possible?}
  \end{mdframed}

  The problem is named after the mathematician Hans Heilbronn (and not after the city in southern Germany). Born into a German-Jewish family, he was dismissed from his position at the University of G\"ottingen in 1933 and emigrated to England~\cite{cassels1977heilbronn}. The problem first appeared in print in Roth's 1951 paper~\cite{Roth1951}, where the unit square was replaced by an arbitrary closed convex region in the plane of positive measure, and the emphasis was placed on the asymptotic behavior of the minimal triangle area rather than on explicit constructions for small values of $n$.

\subsection{Upper bounds}

We denote the minimal triangle area in an optimal configuration of $n$ points in the unit square by $\Delta_n$.
Heilbronn conjectured that $\Delta_n$ should satisfy an asymptotic upper bound of the form
\[
\Delta_n = O\!\left(\frac{1}{n^2}\right).
\]
As we will see in Section~\ref{sec: lower bounds}, this turned out to be incorrect.
Roth observed in~\cite{Roth1951} that the problem admits a simple upper bound,
\[
\Delta_n \le \frac{1}{n-2}.
\]
To see this, note that the convex hull of any configuration of $n$ points can be partitioned into $n-2$ triangles with pairwise disjoint interiors and total area at most $1$, implying that at least one of these triangles must have area no greater than $1/(n-2)$.
Over the following decades, the upper bounds were steadily refined. In 1972, Schmidt~\cite{Schmidt1972} obtained a significant improvement, which was sharpened by Roth~\cite{Roth1972II,Roth1972III} later that year and further strengthened by Koml\'os, Pintz, and Szemer\'edi~\cite{KPS1981} in 1981, establishing what remained the best-known estimate for over forty years.
A major breakthrough came in 2023, when Cohen, Pohoata, and Zakharov~\cite{CPZ2023} proved that for all sufficiently large $n$,
\[
\Delta_n \le n^{-8/7 - 1/2000},
\]
which is currently the strongest known asymptotic upper bound. Figure~\ref{fig:CPZbound} illustrates that this bound provides a reasonably tight envelope for all best-known configurations with $n \ge 6$.

\begin{figure}[h]
  \centering
  \begin{tikzpicture}
  \begin{axis}[
      width=0.65\linewidth,
      height=4.5cm,
      xlabel={$n$},
      ylabel={value},
      xmin=2.8, xmax=16.2,
      ymin=0, ymax=0.55,
      xtick={3,4,...,16},
      ytick={0.1,0.2,...,0.5},
      ymajorgrids,
      grid style={gray!15},
      axis line style={black},
      tick style={black},
      legend style={
          at={(0.9,0.9)},
          anchor=north east,
          draw=none,
          fill=none,
          font=\small
      },
  ]
  
  \addplot[
      very thick,
      gray!20!black,
      mark=*,
      mark size=2.5pt
  ] coordinates {
  (3,0.5)
  (4,0.5)
  (5,{sqrt(3)/9})
  (6,{1/8})
  (7,0.0838)
  (8,{(sqrt(13)-1)/36})
  (9,{(9*sqrt(65)-55)/320})
  (10,0.0465)
  (11,{1/27})
  (12,0.0325)
  (13,0.0270)
  (14,0.0243)
  (15,0.0212)
  (16,{7/341})
  };
  \addlegendentry{Best-known $\Delta_n$}
  
  \addplot[
      thin,
      gray!50,
      mark=square*,
      mark size=2pt
  ] coordinates {
  (3,{pow(3,-(8/7+1/2000))})
  (4,{pow(4,-(8/7+1/2000))})
  (5,{pow(5,-(8/7+1/2000))})
  (6,{pow(6,-(8/7+1/2000))})
  (7,{pow(7,-(8/7+1/2000))})
  (8,{pow(8,-(8/7+1/2000))})
  (9,{pow(9,-(8/7+1/2000))})
  (10,{pow(10,-(8/7+1/2000))})
  (11,{pow(11,-(8/7+1/2000))})
  (12,{pow(12,-(8/7+1/2000))})
  (13,{pow(13,-(8/7+1/2000))})
  (14,{pow(14,-(8/7+1/2000))})
  (15,{pow(15,-(8/7+1/2000))})
  (16,{pow(16,-(8/7+1/2000))})
  };
  \addlegendentry{Cohen--Pohoata--Zakharov bound}
  
  \end{axis}
  \end{tikzpicture}
  \caption{Best-known values of $\Delta_n$ for the Heilbronn problem in the unit square compared with the asymptotic upper bound $n^{-8/7-1/2000}$ of Cohen, Pohoata, and Zakharov~\cite{CPZ2023}.}
  \label{fig:CPZbound}
  \end{figure}

\subsection{Lower bounds}\label{sec: lower bounds}

A simple lower bound of order $c/n^3$ for $\Delta_n$ can be obtained as follows. 
Without regard to the unit-square constraint, consider the $n$ points $(i,i^2)$ for $i=0,\ldots,n-1$. 
No three of these points are collinear, and hence every triangle determined by them has strictly positive area. 
Moreover, each triangle area can be expressed as half the absolute value of a $3\times 3$ determinant formed from the vertices' coordinates (see Section~\ref{sec:original formulation}). 
Since all coordinates are integers and the area is positive, the determinant has absolute value at least $1$, implying that every triangle has area at least $1/2$.
To place the configuration inside the unit square, scale the $x$-coordinates by $1/n$ and the $y$-coordinates by $1/n^2$. 
Areas then scale by $1/n^3$, giving
\[
\Delta_n \ge \frac{1}{2n^3}.
\]
As noted in the appendix of Roth's paper~\cite{Roth1951}, Paul Erd\H{o}s refined this idea by choosing a prime $p$ with $n \le p \le 2n$ and considering the points $(i, i^2 \bmod p)$ for $i=0,\ldots,p-1$. 
This construction preserves the key properties used above, namely integral coordinates and the absence of collinear triples, while having the advantage that the $y$-coordinates are bounded by $p$. 
Consequently, scaling to the unit square reduces areas only by a factor of $1/p^2$, leading to
\[
\Delta_n \ge \frac{1}{2p^2} \ge \frac{1}{8n^2}
\]
for all $n$.

Combined with Heilbronn's conjecture, this would have yielded the correct asymptotic order of magnitude for the problem. 
However, in 1982, following a suggestion by Roth that the conjecture might be false, Koml\'os, Pintz, and Szemer\'edi proved in~\cite{KPS1982} that there exists a constant $c>0$ such that
\[
\Delta_n \ge \frac{c \log n}{n^2}
\]
for infinitely many $n$, using a nonconstructive probabilistic argument.
Since $\log n$ is unbounded, this lower bound eventually exceeds $C/n^2$ for every constant $C$, so the conjectured bound $\Delta_n = O(1/n^2)$ cannot hold.

This gap between existence and construction was closed in 2000 by Bertram-Kretzberg, Hofmeister, and Lefmann~\cite{BHL2000}, who presented a polynomial-time algorithm producing configurations with asymptotic lower bound $\Omega(\log n/n^2)$.

\subsection{Results for small-$n$-configurations}\label{sec:intro-small-n}

For many years, best-known configurations for the Heilbronn problem were documented on \emph{Erich's Packing Center}, a website maintained by Erich Friedman~\cite{FriedmanPackingCenter}.
Since 2026, the records are maintained on the website of Tej Stead~\cite{SteadHeilbronn}, which continues the tables of the Packing Center in a considerably more elaborate and automated form, reflecting the strong renewed interest in the problem.
For each record, the site lists exact coordinates, closed forms where known, critical triangles, history, and proofs, and all submitted configurations are verified automatically in exact arithmetic. We use it as the reference for best-known values.
The first case requiring a nontrivial proof, $n=5$, was resolved in 1991 by Yang, Zhang, and Zeng~\cite{yzz91}, who proved
$\Delta_5 = \sqrt{3}/9$.
In the same work, they disproved a conjecture of Goldberg, who had suggested that optimal configurations for $n \le 7$ arise from affine images of regular $n$-gons, a statement that turns out to hold only for $n=6$.
In 1995, Dress, Yang, and Zeng~\cite{dyz1995} established the optimal value $\Delta_6 = 1/8$.

In 2002, Comellas and Yebra~\cite{comellasyebra2002} used simulated annealing followed by an analytical refinement procedure to discover best-known configurations and exact coordinates for many cases with $n \le 12$.
Their configurations for $n=7$ and $n=8$ have since been proved optimal, their configurations for $n=9$ and $n=10$ are optimal according to all available numerical evidence, including the computations in this paper, and their configuration for $n=12$ remains the best known to date.
The optimality of their configuration for $n=7$, with $\Delta_7 \approx 0.083859$, was proved in 2011 by Zeng and Chen~\cite{ZengChen2011}, who decomposed the problem into 226 nonlinear optimization subproblems.
For $n=9$, Chen, Xu, and Zeng~\cite{chen2017searching} used an extensive search on CPU and GPU clusters, requiring a total of 125 days of wall-clock time, to enclose $\Delta_9$ in a narrow numerical interval.
In 2022, Dehbi and Zeng~\cite{dehbi2022heilbronn} combined numerical search with symbolic computation to prove that the configuration for $n=8$ found by Comellas and Yebra, satisfying
$\Delta_8 = (\sqrt{13}-1)/36$,
is indeed optimal.

For $13 \le n \le 16$, the best-known configurations are due to Peter Karpov ($n=13$), Mark Beyleveld ($n=14,16$), and the author ($n=15$, found in 2026 with a heuristic search). None of them has been formally published. Karpov's configuration is available on his website~\cite{KarpovAscension}, Beyleveld's only through Erich Friedman. We reproduce all of them in Appendix~\ref{app:best-known}.

More recently, Monji, Modir, and Kocuk~\cite{monji2025solvingheilbronntriangleproblem} applied general-purpose MINLP solvers to the problem and solved $n=9$ with Gurobi to global optimality within solver tolerances.
Their final solve took 633 seconds and relies on structural properties of optimal configurations that they establish in precomputations of about 23 hours in total (Section~\ref{sec:comparison-effort}).

Deterministic global optimization has a long record on extremal problems in discrete geometry.
Global QCQP methods determined the largest small octagon~\cite{audet2002largest} and a series of related extremal polygons~\cite{audet2007extremal}, and interval branch-and-bound has produced computer-assisted proofs for covering problems~\cite{cafieri2025covering}.
What has changed since then is efficiency, driven by hardware, solver implementations, relaxations, preprocessing and, not least, the mathematical formulation.
Recent systematic studies by Berthold et al.~\cite{berthold2026global,berthold2026outoftheboxglobaloptimizationpacking} show that off-the-shelf MINLP solvers can match or improve best-known solutions for a range of packing problems.
The present paper illustrates the role of formulation for the Heilbronn problem: Gurobi, which in~\cite{monji2025solvingheilbronntriangleproblem} needed about 23 hours of precomputation for $n=9$, solves our formulation in about eight minutes, and $n=10$ in under seven hours, because the structural information enters through an analytical result rather than a computation.

\subsection{Related work on symmetry}\label{sec:related-symmetry}

Problems of placing points in a region typically combine the Euclidean symmetries of the region with the permutation symmetry of the points, and exploiting both is standard in deterministic global optimization.
General techniques for symmetry in integer programming are surveyed by Margot~\cite{margot2010symmetry}, and the automatic detection and exploitation of symmetry in mathematical programs is studied by Liberti~\cite{liberti2012reformulations}.
For packing equal circles in a square, Locatelli and Raber~\cite{locatelli2002packing} derive geometric properties of optimal configurations to fix coordinates and eliminate symmetric solutions, and Costa, Hansen, and Liberti~\cite{costa2013impact} study the effect of symmetry-breaking constraints on spatial branch-and-bound.
For Lennard-Jones clusters, coordinates are fixed relative to reference particles~\cite{hoare1971physical,kuznetsov2025new}, and for the Thomson problem, branch-and-bound has been combined with analytical arguments and symmetry elimination~\cite{schwartz2013five,kuznetsov2026deterministic}.
Our symmetry breaking follows the same principle but derives the fixings from the boundary structure of optimal Heilbronn configurations (Proposition~\ref{thm:boundary_vertices}).
A feature specific to determinant-based formulations is that the fixed boundary order also determines the orientation of many triangles, which allows sign variables to be fixed a~priori (Section~\ref{sec: sign fixing}).

\subsection{Contributions}\label{sec:contributions}

We make the following contributions.

\paragraph{A compact MINLP formulation with analytically justified structure.}
Our model fixes several coordinates a~priori using the boundary structure of optimal configurations, which also fixes the signs of many triangle areas. It needs no precomputation and solves $n=9$ to global optimality within solver tolerances in about eight minutes on a laptop. It also solves $n=10$, for the first time, in under seven hours. An ablation study quantifies the contribution of each ingredient (Section~\ref{sec:computations}).

\paragraph{A complete justification of the boundary structure.}
The argument previously used for the boundary structure relies on a lemma of Zeng and Chen~\cite{ZengChen2011} that does not address two adjacent corners of the square. We give a self-contained proof without this case and reduce the remaining configurations, by an area-preserving shear, to two further types, which we solve with separate models.

\paragraph{Exact coordinates via an optimize-then-refine framework.}
The numerical solution identifies the critical triangles and the boundary points, which yields a polynomial system that we solve exactly. Every resulting configuration is verified in exact arithmetic over all $\binom{n}{3}$ triangles. For each $n=5,\dots,10$ we obtain exact coordinates of one configuration (for $n=6$, of a one-parameter family), which is proven optimal for $n\le 8$ and optimal within solver tolerances for $n=9$ and $n=10$, matching those of Comellas and Yebra~\cite{comellasyebra2002}.

\paragraph{Structural observations.}
We determine for which optimal configurations the linear independence constraint qualification holds ($n=5,7,9$) and fails ($n=6,8$), conjecture a lower bound of $2n-7$ on the number of critical triangles, and document a clustering of the noncritical triangle areas that symmetry alone does not explain. We also report an observation of Tej Stead (personal communication, 2026) on the cube structure of the configuration for $n=16$.

\paragraph{Reproducibility and data.}
We collect the best-known configurations for $n \le 16$, several of which were never formally published. The source code, data, solution files, and the scripts that generate all tables are available at \url{https://github.com/spiralulam/heilbronn}.

\paragraph{Mathematical status of the results.}
Since the solver works in floating-point arithmetic, we distinguish throughout between rigorous lower bounds (exactly verified configurations), numerical upper bounds (solver dual bounds), and proofs of optimality, which are available in the literature for $n\le 8$ but not for $n=9$ and $n=10$ (Section~\ref{sec:framework-sub}).

\subsection{Paper outline}

Section~\ref{sec:formulation} establishes a formal setting for the Heilbronn problem and derives the boundary structure of optimal configurations that is exploited in the optimization model.
Section~\ref{sec:MIPModel} develops the MINLP formulation and its enhancements.
Section~\ref{sec:computations} reports the computational study.
Section~\ref{sec:framework} presents the optimize-then-refine framework and illustrates it in detail for $n=7$.
Section~\ref{sec:solutions} reports the exact configurations for $n = 5, \ldots, 10$.
Section~\ref{sec:observations} discusses structural observations and open questions.
Section~\ref{sec:outlook} concludes with an outlook.
Appendix~\ref{app:exact-coordinates} lists the exact coordinates for $n = 5, \ldots, 10$, and Appendix~\ref{app:best-known} collects the best-known configurations for $n = 11, \ldots, 16$.

\section{Original formulation and theoretical results}\label{sec:formulation}

This section establishes a basic formulation of the Heilbronn problem in Section~\ref{sec:original formulation}, which already suffices to prove its solvability.
The results of Section~\ref{sec: boundary configuration} on the boundary structure of optimal configurations are not only of theoretical interest but provide structure that is exploited in the optimization model of Section~\ref{sec:MIPModel}.

\subsection{The original formulation and its solvability}\label{sec:original formulation}

Let $n\ge 3$ denote the number of points $p_1=(x_1,y_1),\ldots,p_n=(x_n,y_n)$ we want to place in $\R^2$, $I=\{1,\ldots,n\}$, and 
\[
T\ =\ \{(i,j,k)|\ 1\le i<j<k\le n\}\subset \N^3
\]
the set of all triangles, where each triangle $t=(i,j,k)\in T$ is identified with the indices of the defining points $p_i$, $p_j$ and $p_k$.
By construction, $|T|=\binom{n}{3}$.
\begin{remark}
  The enumeration of the points heavily influences the symmetry breaking and therefore 
   becomes a crucial modeling choice, as discussed in Section~\ref{sec: symmetry breaking}.
\end{remark}

For a triangle $t=(i,j,k)\in T$, the \emph{signed area}
\[
A^\pm_t=A^\pm_{ijk} := \frac{1}{2}\det\!\begin{pmatrix}
  x_i & y_i & 1\\
  x_j & y_j & 1\\
  x_k & y_k & 1
  \end{pmatrix}
= \frac12\left(x_i y_j + x_j y_k + x_k y_i - x_i y_k - x_j y_i - x_k y_j\right)
\]
is positive if $p_i,p_j,p_k$ are ordered counterclockwise, negative if they are ordered clockwise, and zero if they are collinear, and $A_t:=|A^\pm_t|$ is the area of the triangle.
The sign information plays a key role in the optimization model of Section~\ref{sec:MIPModel}.
With $S=[0,1]^2$, the Heilbronn problem is the optimization problem
\[
P_\Delta:\qquad
\max_{x,y}\ \min_{t\in T} A_t
\quad
\text{s.t. } (x_i,y_i)\in S,\ i\in I,
\]
and $\Delta_n$ is its optimal value.
All further specifications of this problem amount to modeling decisions, of which there are many and which must be made carefully.
As stated, $P_\Delta$ is a nonsmooth, nonconvex optimization problem over a box, and it has the following properties.

\begin{proposition}\label{prop:solvability-positive-optval}
  The optimization problem $P_\Delta$ admits a solution for every $n$, and its optimal value $\Delta_n$ is strictly positive.
  \end{proposition}
  
  \begin{proof}
  Existence of a solution follows from the Weierstraß theorem. 
  The objective function is continuous, and the feasible set is nonempty, closed, and bounded, hence compact. 
  Therefore, the maximum is attained. Strict positivity of $\Delta_n$ follows, for example, from Erd\H{o}s' construction discussed in Section~\ref{sec: lower bounds}, which provides an explicit positive lower bound.
  \end{proof}
  
  \begin{remark}
  In~\cite{Roth1951}, Roth argues that $P_\Delta$ is solvable because its feasible set \enquote{is a closed convex region}. 
  However, closedness and convexity alone do not guarantee existence of a maximizer. Compactness (or an equivalent coercivity argument) is required.
  \end{remark}
  
  The following corollaries are immediate consequences of Proposition~\ref{prop:solvability-positive-optval}, since otherwise the configuration would contain a triangle of zero area.

  \begin{corollary}
  An optimal configuration for the Heilbronn problem does not contain two coincident points.
  \end{corollary}
  
  \begin{corollary}
  An optimal configuration for the Heilbronn problem does not contain three collinear points.
  \end{corollary}

\subsection{Boundary structure of optimal configurations}\label{sec: boundary configuration}

The symmetry breaking of Section~\ref{sec:MIPModel} fixes five points on the boundary of the unit square $S=[0,1]^2$ in a prescribed pattern.
This subsection shows that this is justified.
We first fix notation.
A \emph{configuration} is a set $P=\{p_1,\dots,p_n\}\subseteq S$ of $n$ points.
We let $\Delta(P)$ denote the smallest area of a triangle spanned by three points of $P$ and call $P$ \emph{optimal} if $\Delta(P)=\Delta_n$.
The objective and the constraints of $P_\Delta$ do not depend on the order of the points, but the symmetry breaking of Section~\ref{sec:MIPModel} does.
We therefore call a configuration together with a fixed order $(p_1,\dots,p_n)$ of its points a \emph{labeled configuration}, and a change of this order a \emph{relabeling}.
Finally, let
\[
K := \operatorname{conv}(P)
= \Bigl\{\sum_{i=1}^{n}\lambda_i\, p_i \;\Big|\; \lambda_i\ge 0,\;\sum_{i=1}^{n}\lambda_i=1\Bigr\}
\]
denote its convex hull.
We write $A=(0,0)$, $B=(1,0)$, $C=(1,1)$, $D=(0,1)$ for the corners of $S$ and $\sigma_{\mathrm b}$, $\sigma_{\mathrm r}$, $\sigma_{\mathrm t}$, $\sigma_{\ell}$ for its bottom, right, top, and left edge (Figure~\ref{fig:boundary-ideas}a).

\begin{figure}[ht]
  \centering
  \begin{tikzpicture}[scale=3.1, line join=round,
      pt/.style={circle, fill=black, inner sep=0.9pt},
      opt/.style={circle, draw=black, fill=white, inner sep=0.9pt},
      lab/.style={font=\scriptsize}]
    \begin{scope}
      \draw[thick] (0,0) rectangle (1,1);
      \fill[gray!18] (0,0.25) -- (0.35,0) -- (1,0.3) -- (0.6,1) -- (0,0.75) -- cycle;
      \draw[gray!60] (0,0.25) -- (0.35,0) -- (1,0.3) -- (0.6,1) -- (0,0.75) -- cycle;
      \node[lab] at (0.45,0.5) {$K$};
      \foreach \p/\l/\a in {(0,0.25)/p_1/left, (0.35,0)/p_2/below, (1,0.3)/p_3/right, (0.6,1)/p_4/above, (0,0.75)/p_5/left}
        \node[pt, label={[lab]\a:$\l$}] at \p {};
      \node[pt] at (0.55,0.45) {};
      \node[lab, anchor=north east] at (0,0) {$A$};
      \node[lab, anchor=north west] at (1,0) {$B$};
      \node[lab, anchor=south west] at (1,1) {$C$};
      \node[lab, anchor=south east] at (0,1) {$D$};
      \node[lab, below=7pt] at (0.72,0) {$\sigma_{\mathrm b}$};
      \node[lab, right=7pt] at (1,0.7) {$\sigma_{\mathrm r}$};
      \node[lab, above=7pt] at (0.28,1) {$\sigma_{\mathrm t}$};
      \node[lab, left=7pt] at (0,0.5) {$\sigma_{\ell}$};
      \node[font=\small] at (0.5,-0.32) {(a)};
    \end{scope}
    \begin{scope}[xshift=1.5cm]
      \draw[thick] (0,0) rectangle (1,1);
      \fill[gray!18] (0.45,0) -- (1,0.62) -- (0.45,1) -- (0,0.62) -- cycle;
      \draw[gray!60] (0.45,0) -- (1,0.62) -- (0.45,1) -- (0,0.62) -- cycle;
      \node[lab] at (0.5,0.5) {$K$};
      \draw[red!70!black, dashed, thick] (0.252,-0.072) -- (1.135,0.249) -- (0.814,1.132) -- (-0.069,0.811) -- cycle;
      \foreach \p in {(0.45,0),(1,0.62),(0.45,1),(0,0.62)} \node[pt] at \p {};
      \node[lab, red!70!black] at (0.62,0.25) {area $0.88$};
      \node[font=\small] at (0.5,-0.32) {(b)};
    \end{scope}
    \begin{scope}[xshift=3cm]
      \draw[thick] (0,0) rectangle (1,1);
      \fill[gray!18] (0.7,0) -- (1,0.3) -- (1,0.6) -- (0.3,1) -- (0,0.55) -- cycle;
      \draw[gray!60] (0.7,0) -- (1,0.3) -- (1,0.6) -- (0.3,1) -- (0,0.55) -- cycle;
      \node[lab] at (0.5,0.5) {$K$};
      \draw[red!70!black, dashed, thick] (0,0.175) -- (1,-0.075) -- (1,0.825) -- (0,1.075) -- cycle;
      \node[pt, label={[lab]above:$(t,1)$}] at (0.3,1) {};
      \node[pt, label={[lab]below:$(b,0)$}] at (0.7,0) {};
      \draw[densely dotted] (0.3,1) -- (0.3,0);
      \draw[densely dotted] (0.7,1) -- (0.7,0);
      \node[lab, red!70!black] at (0.5,0.12) {area $0.9$};
      \node[font=\small] at (0.5,-0.32) {(c)};
    \end{scope}
  \end{tikzpicture}
  \caption{(a)~Notation: corners $A,B,C,D$, edges $\sigma_{\mathrm b},\sigma_{\mathrm r},\sigma_{\mathrm t},\sigma_\ell$, and the convex hull $K$ of a configuration with the boundary pattern of Definition~\ref{def:pattern}.
  (b)~Idea of Lemma~\ref{lem:contact}(B): if every edge of $S$ touches $K$ in a single point, a slightly rotated rectangle (dashed) encloses $K$ with area smaller than~$1$. Here opposite contacts face each other, so $K$ satisfies~(C) but violates~(B).
  (c)~Idea of Lemma~\ref{lem:contact}(C): if the top contact lies to the left of the bottom contact, tilting the top and bottom edges yields a covering parallelogram (dashed) of area smaller than~$1$, although it extends beyond~$S$. Here $K$ contains an edge on $\sigma_{\mathrm r}$, so it satisfies~(B) but violates~(C).}
  \label{fig:boundary-ideas}
\end{figure}

\begin{definition}[Boundary pattern]\label{def:pattern}
A labeled configuration $p_1,\dots,p_n$ with $n\ge 5$ has the \emph{boundary pattern} if
$x_1=x_5=0$, $y_2=0$, $x_3=1$, $y_4=1$, $y_1\le y_5$, $x_2\le x_4$, and $x_6\le x_7\le\dots\le x_n$.
\end{definition}

\begin{lemma}[Minimal covering parallelogram]\label{lem:min_parallelogram}
If $P\subseteq S$ is an optimal configuration, then $S$ is a minimum-area parallelogram containing $K$.
\end{lemma}

\begin{proof}
Assume for contradiction that there exists a parallelogram $S' \supseteq K$
with $\operatorname{area}(S') < 1$.
Since every parallelogram is an affine image of the unit square,
there exists an affine bijection $\varphi: S' \to S$
satisfying
\[
|\det \varphi| = \frac{\operatorname{area}(S)}{\operatorname{area}(S')}
= \frac{1}{\operatorname{area}(S')} > 1.
\]
Because $K \subseteq S'$, we have $\varphi(P) \subseteq S$.
Every triangle spanned by points of $P$ has its area multiplied by $|\det\varphi|$, hence
$\Delta(\varphi(P)) = |\det \varphi|\, \Delta(P) > \Delta(P)$,
contradicting the optimality of $P$.
\end{proof}

The next lemma collects three contact properties of a minimum-area covering parallelogram. They are purely geometric.

\begin{lemma}[Contact properties]\label{lem:contact}
Let $K\subseteq S$ be a convex polygon such that $S=[0,1]^2$ is a minimum-area parallelogram containing $K$. Then:
\begin{enumerate}
\item[\textup{(A)}] Each edge of $S$ meets $K$, and each contact $K\cap\sigma$ is a vertex or an edge of $K$.
\item[\textup{(B)}] At least one edge of $S$ contains an entire edge of $K$.
\item[\textup{(C)}] The projections to the $x$-axis of $K\cap\sigma_{\mathrm b}$ and $K\cap\sigma_{\mathrm t}$ intersect, and so do the projections to the $y$-axis of $K\cap\sigma_{\ell}$ and $K\cap\sigma_{\mathrm r}$.
\end{enumerate}
\end{lemma}

\begin{proof}
\emph{(A).} If $\sigma_{\mathrm b}\cap K=\emptyset$, then $\delta:=\min_{K}y>0$, and the rectangle $[0,1]\times[\delta,1]$ contains $K$ and has area $1-\delta<1$. The other edges are analogous. The line through an edge of $S$ is then a supporting line of $K$, so the contact is a face of $K$.

\emph{(B).} The idea is illustrated in Figure~\ref{fig:boundary-ideas}b. For $(\alpha,\beta)$ close to $(0,0)$ let $\ell_1=x-\beta y$ and $\ell_2=y-\alpha x$, and let $S(\alpha,\beta)$ be the parallelogram bounded by the lines $\ell_1=\min_K\ell_1$, $\ell_1=\max_K\ell_1$, $\ell_2=\min_K\ell_2$, and $\ell_2=\max_K\ell_2$. It contains $K$, $S(0,0)=S$, and
\[
\operatorname{area}S(\alpha,\beta)=\frac{w_1(\beta)\,w_2(\alpha)}{1-\alpha\beta},
\qquad w_1(\beta)=\max_K\ell_1-\min_K\ell_1,\quad w_2(\alpha)=\max_K\ell_2-\min_K\ell_2 .
\]
Suppose that every contact is a single vertex, say $(b_1,0)$, $(t_1,1)$, $(0,l_2)$, $(1,r_2)$. These vertices remain the unique optimizers of $\pm\ell_1,\pm\ell_2$ for small $(\alpha,\beta)$, so $w_1(\beta)=1-c\beta$ and $w_2(\alpha)=1-d\alpha$ with $c=r_2-l_2$ and $d=t_1-b_1$. Along $(\alpha,\beta)=(s,-s)$ we obtain
\[
\operatorname{area}S(s,-s)=\frac{(1+cs)(1-ds)}{1+s^2},
\]
which is smaller than $1$ for a suitable small $s\neq 0$: if $c\ne d$, the linear term $(c-d)s$ can be made negative, and if $c=d$, the value is $(1-c^2s^2)/(1+s^2)<1$. This contradicts minimality.

\emph{(C).} The idea is illustrated in Figure~\ref{fig:boundary-ideas}c. For $\beta=0$, $S(\alpha,0)$ has the vertical edges of $S$ and area $w_2(\alpha)$. The convex function $w_2$ is minimal at $\alpha=0$, so $0\in\partial w_2(0)$. If $[b^-,b^+]$ and $[t^-,t^+]$ denote the $x$-ranges of $K\cap\sigma_{\mathrm b}$ and $K\cap\sigma_{\mathrm t}$, then $\partial w_2(0)=[b^--t^+,\ b^+-t^-]$, and $0\in\partial w_2(0)$ says that the two ranges intersect. The case $\alpha=0$ gives the statement for the vertical edges.
\end{proof}

\begin{remark}
Part~(B) is closely related to the classical fact that a minimum-area rectangle enclosing a convex polygon has a side containing an edge of the polygon~\cite{freeman1975determining}. Part~(C) generalizes conditions (a)--(c) in case~1 of Lemma~2 of Zeng and Chen~\cite{ZengChen2011}: contacts on opposite edges face each other, and two single-point contacts lie on a line parallel to the other two edges. That lemma covers the cases of no common vertex, one common vertex, and two opposite common vertices, but not two adjacent common vertices. For the fixed unit square this case must be handled explicitly, which is done in Lemma~\ref{lem:shear} and Theorem~\ref{thm:case-split} below.
\end{remark}

A corner of $S$ lies in $K$ if and only if it is a vertex of $K$, that is, a point of $P$.
We call two corners \emph{adjacent} if they share an edge of $S$.

\begin{proposition}[Configurations without adjacent corners]\label{thm:boundary_vertices}
Let $n\ge 5$ and let $P\subseteq S$ be an optimal configuration that does not contain two adjacent corners of $S$.
Then at least five points of $P$ lie on $\partial S$, and the image of $P$ under a suitable symmetry of $S$ can be labeled so that it has the boundary pattern.
\end{proposition}

\begin{proof}
By Lemma~\ref{lem:min_parallelogram}, Lemma~\ref{lem:contact} applies to $S$ and $K$. By assumption, $P$ contains at most two corners, and if it contains two, they are opposite.

\emph{No corner.} By (A), every contact lies in the relative interior of its edge, which gives four distinct boundary points. By (B), some edge contains an entire edge of $K$ and hence a second point. We rotate $S$ so that this edge becomes $\sigma_\ell$ and call its two points $p_1$ and $p_5$ with $y_1<y_5$, and the contacts on $\sigma_{\mathrm b}$, $\sigma_{\mathrm r}$, $\sigma_{\mathrm t}$ $p_2$, $p_3$, $p_4$.

\emph{One corner.} After a rotation, $A\in P$ and $B,C,D\notin P$. If $\sigma_{\mathrm b}$ met $K$ only in $A$, its $x$-range would be $\{0\}$, and by (C) the contact on $\sigma_{\mathrm t}$ would contain a point with $x=0$, that is, $D$, which is excluded. Hence $\sigma_{\mathrm b}$ contains an edge $A\,(\beta,0)$ of $K$ with $0<\beta<1$, and in the same way $\sigma_\ell$ contains an edge $A\,(0,\lambda)$ with $0<\lambda<1$. Together with a point on $\sigma_{\mathrm r}$ and a point on $\sigma_{\mathrm t}$, which differ from each other and from the previous ones because $B,C,D\notin P$, this gives five boundary points. We set $p_1=A$, $p_2=(\beta,0)$, $p_5=(0,\lambda)$, and let $p_3$ and $p_4$ be the points on $\sigma_{\mathrm r}$ and $\sigma_{\mathrm t}$.

\emph{Two opposite corners.} After a rotation, $A,C\in P$ and $B,D\notin P$. The argument of the previous case, applied to all four edges, shows that every edge contains an edge of $K$, which gives the six boundary points $A$, $(\beta,0)$, $(1,\rho)$, $C$, $(\tau,1)$, $(0,\lambda)$. We set $p_1=A$, $p_2=(\beta,0)$, $p_3=(1,\rho)$, $p_4=(\tau,1)$, $p_5=(0,\lambda)$.

In all cases $p_1,\dots,p_5$ are distinct, $x_1=x_5=0$, $y_2=0$, $x_3=1$, $y_4=1$, and $y_1\le y_5$.
If $x_2>x_4$, we apply the reflection $(x,y)\mapsto(x,1-y)$, which preserves $\sigma_\ell$ and exchanges $\sigma_{\mathrm b}$ and $\sigma_{\mathrm t}$, and relabel the reflected points $p_5,p_4,p_3,p_2,p_1$ as $p_1,\dots,p_5$. This preserves all conditions above and exchanges $x_2$ and $x_4$.
Finally, we label the remaining points by nondecreasing $x$-coordinate.
\end{proof}

It remains to treat configurations with two adjacent corners. Such configurations do occur: the known optimal configuration for $n=7$ contains $C$ and $D$ (Table~\ref{tab:n7-numerical}).

\begin{lemma}[Shear normalization]\label{lem:shear}
Let $P\subseteq S$ be an optimal configuration that contains the corners $(1,1)$ and $(0,1)$. Then there is an optimal configuration $P'\subseteq S$ that contains $(1,1)$, $(0,1)$, a point on the bottom edge, and a point $(0,\lambda)$ with $\lambda<1$.
\end{lemma}

\begin{proof}
By Lemma~\ref{lem:min_parallelogram} and Lemma~\ref{lem:contact}(A), $P$ contains a point $(b,0)$ on $\sigma_{\mathrm b}$. If $P$ already contains a point $(0,\lambda)$ with $\lambda<1$, we take $P'=P$. Otherwise every point $p\in P$ with $y_p<1$ has $x_p>0$, and we set
\[
s^\ast := \min\Bigl\{\frac{x_p}{1-y_p} \;\Big|\; p\in P,\ y_p<1\Bigr\} \in (0,1],
\]
where $s^\ast\le 1$ follows from the point $(b,0)$. The shear $\psi(x,y)=(x-s^\ast(1-y),\,y)$ has determinant $1$ and therefore preserves all triangle areas. It fixes every point with $y=1$, in particular $C$ and $D$, and it does not change $y$-coordinates. By the choice of $s^\ast$, every point $p$ with $y_p<1$ satisfies $0\le x_p-s^\ast(1-y_p)\le x_p\le 1$, so $P'=\psi(P)\subseteq S$ is optimal, contains $(b-s^\ast,0)$ on $\sigma_{\mathrm b}$, and contains a point on $\sigma_\ell$ below $D$, namely the image of a minimizer of $s^\ast$.
\end{proof}

\begin{theorem}[Structure of optimal configurations]\label{thm:case-split}
Let $n\ge 5$. There exists an optimal configuration that, after a symmetry of $S$ and a relabeling, has one of the following types:
\begin{enumerate}
\item[\textup{I.}] it has the boundary pattern of Definition~\ref{def:pattern},
\item[\textup{II.}] it contains the corners $(1,1)$ and $(0,1)$, a point $(0,\lambda)$ with $\lambda<1$, and a different point $(b,0)$,
\item[\textup{III.}] it contains the corners $(0,0)$, $(1,1)$, and $(0,1)$.
\end{enumerate}
\end{theorem}

\begin{proof}
Let $P$ be optimal. If $P$ contains no two adjacent corners, Proposition~\ref{thm:boundary_vertices} gives type~I. Otherwise, after a rotation, $C,D\in P$, and Lemma~\ref{lem:shear} yields an optimal configuration containing $C$, $D$, a point $(0,\lambda)$ with $\lambda<1$, and a point $(b,0)$. If these two points differ, the configuration has type~II. If they coincide, the point is $A$, and the configuration has type~III.
\end{proof}

In each type, the prescribed boundary points determine the orientation of many triangles, which allows the corresponding sign variables to be fixed (Sections~\ref{sec: sign fixing} and~\ref{sec:type-models}).
Consequently, $\Delta_n$ is the largest of the optimal values of three restricted problems, one for each type.
The model of Section~\ref{sec:MIPModel} covers type~I. Types~II and~III are covered by two variants with different fixings, which we solve separately (Section~\ref{sec:computations}).
A configuration can have several types. For example, the optimal configuration for $n=7$ in Table~\ref{tab:n7-numerical} has the boundary pattern, so it has type~I, and it contains $C$, $D$, $(0,y_1)$, and $(x_2,0)$, so it also has type~II.

\section{Mixed-integer nonlinear optimization models}\label{sec:MIPModel}

In Section~\ref{sec:MIPmodelVanilla} we present a baseline MINLP formulation that is an exact reformulation of the nonsmooth problem $P_\Delta$ without any additional structure. It solves small instances to global optimality within solver tolerances but does not solve $n=7$ within one hour.
Section~\ref{sec:MIPmodelImproved} then develops several enhancements, in particular a symmetry-breaking strategy based on the boundary structure of Section~\ref{sec: boundary configuration} that, as a byproduct, also fixes the signs of certain triangle areas.
The resulting final model is collected in Section~\ref{sec:final-model}. Section~\ref{sec:type-models} describes the variants for the configuration types~II and~III of Theorem~\ref{thm:case-split}, and Section~\ref{sec:comparison-monji} compares our formulation with that of Monji, Modir, and Kocuk~\cite{monji2025solvingheilbronntriangleproblem}.

\subsection{A first MINLP formulation}\label{sec:MIPmodelVanilla}

For each triangle $t=(i,j,k) \in T$ we introduce a continuous auxiliary variable
$A_t^{\pm} \in [-\tfrac{1}{2}, \tfrac{1}{2}]$ representing the signed area
\[
A_t^{\pm}
=
\frac{1}{2}\left(
x_i y_j + x_j y_k + x_k y_i
-
x_i y_k - x_j y_i - x_k y_j
\right).
\]
To eliminate the minimum over triangles in the objective, we introduce an auxiliary variable
$z$ and impose
\[
z \le |A_t^{\pm}|, \qquad t \in T.
\]
This formulation remains nonsmooth because of the absolute value.
To handle it, we introduce a binary variable $b_t\in\{0,1\}$ for each triangle $t$
that selects the sign of $A_t^{\pm}$,
and replace the absolute-value constraint by
\[
z \le (2b_t - 1)\, A_t^{\pm}, \qquad t \in T.
\]
Maximizing $z$ forces the model to choose the sign that yields the larger of $-A_t^\pm$ and $A_t^\pm$, which gives the absolute value
$|A_t^\pm|=\max(A_t^\pm,-A_t^\pm)$. For fixed coordinates, the optimal choice of the $b_t$ and of $z$ therefore yields $z=\min_{t\in T}|A_t^\pm|$, so the following model is equivalent to $P_\Delta$: both have the same optimal value and the same optimal coordinates.

Our first baseline model is therefore

\begin{align*}
P_\Delta^0:\qquad
\max_{z,x,y}\quad & z \\
\text{s.t.}\quad
& (x_i,y_i)\in[0,1]^2, \quad i\in I,\\
& A_t^{\pm}
=
\frac{1}{2}\left(
x_i y_j + x_j y_k + x_k y_i
-
x_i y_k - x_j y_i - x_k y_j
\right),
\quad t\in T,\\
& z \le (2b_t - 1)\, A_t^{\pm}, \quad t\in T,\\
& A_t^{\pm}\in\left[-\frac{1}{2}, \frac{1}{2}\right],\\
& b_t \in \{0,1\}, \quad t\in T,\\
&  0 \le z \le \frac{1}{2}.
\end{align*}

This formulation is an MINLP and can in principle be solved by any global MINLP solver.
In our experiments (Section~\ref{sec:computations}), it solves $n=5$ and $n=6$ in 0.8 and 216 seconds but does not terminate within one hour for $n=7$, where the relative gap is still 78\%.

\subsection{Enhancements}\label{sec:MIPmodelImproved}

We enhance the baseline model with three groups of improvements.

\subsubsection{Bounds and variable substitution}\label{sec:bounds-subst}

Removing a point from an optimal configuration with $n$ points does not decrease the minimum triangle area, so $\Delta_n\le\Delta_{n-1}$, and we add the constraint $z \le \Delta_{n-1}$.
For $n\le 9$ we use the proven values of $\Delta_{n-1}$. For $n=10$, where no proof for $\Delta_9$ is available, we use the dual bound $0.0548818$ computed for $n=9$ (Section~\ref{sec:main-results}, rounded up). The cap only removes configurations whose minimum area exceeds this bound, but our result for $n=10$ therefore also rests on the numerical result for $n=9$.

Additionally, we introduce auxiliary variables $w_{ij} \in [0,1]$ with $w_{ij} = x_i y_j$ for all $i,j\in I$. The signed areas then become linear in $w$,
\[
A^\pm_{ijk}=\tfrac12\left(w_{ij}+w_{jk}+w_{ki}-w_{ik}-w_{ji}-w_{kj}\right),
\]
and all nonconvexity of the area definitions is concentrated in the $n^2$ bilinear equations $w_{ij}=x_iy_j$.
Global solvers handle such equations by spatial branch-and-bound~\cite{belotti2013mixed,burer2012nonconvex}. For two bounded variables, the McCormick inequalities~\cite{mccormick1976computability} describe the convex hull of the graph of the bilinear term over the box. They are a relaxation of $w_{ij}=x_iy_j$, not an exact linearization, and spatial branching tightens them.
Since each product occurs in the $n-2$ triangles containing both $p_i$ and $p_j$, tightening the relaxation of one $w_{ij}$ strengthens all of them at once.
Monji, Modir, and Kocuk~\cite{monji2025solvingheilbronntriangleproblem} use the same substitution, so it does not explain the difference between the two formulations. In the ablation of Section~\ref{sec:ablation} its effect turns out to be negligible, presumably because Gurobi introduces such auxiliary variables itself.

\subsubsection{Symmetry breaking}\label{sec: symmetry breaking}

The primary obstacle is the enormous symmetry: the dihedral group $D_4$ acts on coordinates (8 symmetries: 4 rotations, 4 reflections) and the symmetric group $S_n$ permutes point labels, yielding a combined symmetry group of order $8 \cdot n!$. For $n=9$, this is about 2.9 million equivalent copies of every configuration.
By Theorem~\ref{thm:case-split}, some optimal configuration has type~I, II, or~III. The model of this section covers type~I and requires the boundary pattern of Definition~\ref{def:pattern}. This amounts to three steps.

\paragraph{Step 1: Choice of the left edge.}
An edge carrying two boundary points is designated as the left edge, and the constraint $x_2 \le x_4$ selects one of the two configurations related by the reflection $(x,y)\mapsto(x,1-y)$.

\paragraph{Step 2: Canonical boundary labeling.}
The five boundary points $p_1, \ldots, p_5$ are labeled counterclockwise starting from the left edge: $p_1$ (left, lower), $p_2$ (bottom), $p_3$ (right), $p_4$ (top), $p_5$ (left, upper). The constraint $y_5 \ge y_1$ fixes the order of the two points on the left edge.

\paragraph{Step 3: Interior point ordering.}
The remaining points $p_6, \ldots, p_n$ are ordered by nondecreasing $x$-coordinate.

These steps fix five coordinate variables outright and remove most, but not all, of the symmetry.
If several edges carry two boundary points, or if the configuration contains a corner of $S$, several rotated or reflected copies remain feasible. Points with equal $x$-coordinates can still be interchanged in Step~3.
Figure~\ref{fig:symbreak} illustrates the resulting configuration structure, and Table~\ref{tab:copies} reports how many copies of each best-known configuration remain feasible.

\begin{figure}[ht]
  \centering
  \begin{tikzpicture}[
      scale=3.9,
      point/.style={circle, fill=black, inner sep=1.1pt},
      lab/.style={font=\small},
      order/.style={-Latex, black, dashed, line width=0.8pt},
      guide/.style={densely dashed, gray!60, line width=0.5pt},
      ineq/.style={gray!70, line width=0.95pt}
  ]
  
  \draw[line width=1pt] (0,0) rectangle (1,1);
  
  \coordinate (p1) at (0,0.20);
  \coordinate (p2) at (0.25,0);
  \coordinate (p3) at (1,0.30);
  \coordinate (p4) at (0.60,1);
  \coordinate (p5) at (0,0.75);
  
  \coordinate (p6) at (0.35,0.35);
  \coordinate (p7) at (0.55,0.45);
  \coordinate (p8) at (0.75,0.4);
  
  \draw[order] (p1) to[bend right=12] (p2);
  \draw[order] (p2) to[bend right=12] (p3);
  \draw[order] (p3) to[bend right=12] (p4);
  \draw[order] (p4) to[bend left=12] (p5);

  \draw[order] (p5) to[bend right=12] (p6);
  
  \draw[order] (p6) to[bend right=12] (p7);
  \draw[order] (p7) to[bend left=12] (p8);
  
  \foreach \i/\p/\pos in {1/p1/left,2/p2/below,3/p3/right,4/p4/above,5/p5/left,6/p6/below,7/p7/above,8/p8/right}{
      \node[point, label={[lab, inner sep=1.5pt]\pos:$p_{\i}$}] at (\p) {};
  }
  
  \draw[guide] (p2) -- ($(p2)+(0,1)$);
  \draw[guide] (p4) -- ($(p4)-(0,1)$);
  
\node[font=\small, gray!70] at ($(p2)!0.5!(p4) + (0,0.125)$) {$x_2 \le x_4$};  
\end{tikzpicture}
  \caption{Symmetry breaking used in the optimization model.
  Five points are fixed on the boundary in counterclockwise order ($x_1=x_5=0$, $y_2=0$, $x_3=1$, $y_4=1$, with $y_1\le y_5$), the remaining points are ordered by nondecreasing $x$-coordinate, and additionally $x_2 \le x_4$.}
  \label{fig:symbreak}
\end{figure}

\begin{remark}
  This symmetry-breaking approach also yields correct results for $n<5$, although the theoretical justification does not stem from Proposition~\ref{thm:boundary_vertices}
  but from a separate investigation of the cases $n=3$ and $n=4$. 
\end{remark}

\begin{table}[ht]
  \centering
  \small
  \begin{tabular}{c c c c c c}
    \toprule
    $n$ & 5 & 7 & 8 & 9 & 10 \\
    \midrule
    order $8\cdot n!$ of the symmetry group & 960 & 40{,}320 & 322{,}560 & 2{,}903{,}040 & 29{,}030{,}400 \\
    feasible copies of the best-known configuration & 1 & 3 & 7 & 24 & 8 \\
    \bottomrule
  \end{tabular}
  \caption{Number of labeled copies of the best-known configuration that are feasible for the final model $P^\star_\Delta$, counted in exact arithmetic by applying all eight symmetries of the square and all labelings. For $n=6$ the optimal configurations form a one-parameter family, so the count is infinite.}
  \label{tab:copies}
\end{table}

The counts in Table~\ref{tab:copies} were computed in exact arithmetic. For each of the eight symmetries $g$ of the square, we enumerate all labelings of $g(P)$, where $P$ is the best-known configuration, keep those that satisfy all constraints of $P^\star_\Delta$ (allowing every order among points with equal $x$-coordinates), and count the distinct labeled configurations. The script is part of the accompanying repository.

The remaining copies have three sources:
(i)~the choice of the left edge among the edges carrying two boundary points,
(ii)~the choice of $p_2$, $p_3$, or $p_4$ when an edge carries a second point, which then joins the points ordered by $x$-coordinate, and
(iii)~ties in an order constraint, including $x_2\le x_4$, which otherwise selects one of two mirror images.
All three depend on the boundary structure of the particular configuration, so the counts follow no simple pattern in~$n$.

\begin{example}
For $n=7$, only source~(i) occurs. Three edges carry two points (Table~\ref{tab:n7-numerical}): the left edge ($(0,0.1352)$ and $D$), the right edge ($(1,0.1808)$ and $C$), and the top edge ($C$ and $D$). Each of them can become the left edge, which gives the three copies: the solution in Table~\ref{tab:n7-numerical} itself, its reflection $(x,y)\mapsto(1-x,y)$, and its reflection $(x,y)\mapsto(1-y,1-x)$ in the anti-diagonal.

For $n=9$, all three sources occur. The configuration of Comellas and Yebra (Table~\ref{tab:exact-89}) carries two points on every edge and is itself symmetric under reflection in the anti-diagonal, so the eight symmetries of the square produce only four different images. In each image, $p_2$, $p_3$, and $p_4$ can each be chosen in two ways, which gives $2\cdot2\cdot2=8$ labelings of the boundary pattern (source~(ii)). The constraint $x_2\le x_4$ keeps six of them in two of the images and four in the other two, and in each image two of the kept labelings satisfy $x_2=x_4=\tfrac38-\tfrac{\sqrt{65}}{40}$, so both mirror images pass (source~(iii)). In the first two images, two of the remaining points also share an $x$-coordinate, which doubles two labelings (source~(iii)). This gives $8+8+4+4=24$ copies.
\end{example}

\subsubsection{Sign fixing}\label{sec: sign fixing}

The boundary pattern allows us to fix the signs for two groups of triangles:
\begin{enumerate}
  \item All triangles in the set $T^+=\{(i,j,k)\in T\mid k\le 5\}$ have their three vertices on the boundary, and $p_1,\dots,p_5$ occur in counterclockwise order along $\partial S$. Since no three points of an optimal configuration are collinear, $A_t^\pm> 0$ for all $t\in T^+$.
  \item Since $p_1$ and $p_5$ are located on the left edge with $y_5\ge y_1$, we have $A^\pm_{15k}=-\tfrac12(y_5-y_1)x_k\le 0$ for all $t\in T^-:=\{(i,j,k)\in T\mid i=1,\ j=5\}$.
\end{enumerate}

This allows us to add the constraints
\[
b_t=1,\ t\in T^+\qquad\text{and}\qquad b_t=0,\ t\in T^-
\]
and thereby fix $\binom{5}{3}+(n-5)$ binary variables in advance.
\begin{remark}
  It is also valid to introduce binary variables only for the triangles whose sign is not fixed and to impose $z\le A^\pm_t$ or $z\le -A^\pm_t$ directly for the others. Section~\ref{sec:ablation} compares both variants.
\end{remark}

\subsection{The final model}\label{sec:final-model}

Combining all ingredients from the previous subsection we arrive at the formulation~$P_\Delta^\star$ displayed in Figure~\ref{fig:final-model}.
The colored boxes highlight the individual enhancement groups. The core model corresponds to the baseline formulation~$P_\Delta^0$ expressed in the substituted variables.
Unlike $P_\Delta^0$, the model $P_\Delta^\star$ is not equivalent to $P_\Delta$. It is a restriction to labeled configurations of type~I, so its optimal value is at most $\Delta_n$. By Theorem~\ref{thm:case-split}, $\Delta_n$ is attained by $P_\Delta^\star$ or by one of the variants of Section~\ref{sec:type-models}.

\begin{figure}[h]
    \centering
    
    \begin{align*}
      P_\Delta^\star:\qquad \max_{z,x,y}\quad & z \\[2pt]
      \text{s.t.}\quad
      & \tikzmark{core-tl}(x_i,y_i)\in[0,1]^2,\quad i\in I,\\
      & A_{ijk}^{\pm}\ =\ \frac{1}{2}\left(
        w_{ij} + w_{jk} + w_{ki} - w_{ik} - w_{ji} - w_{kj}
        \right),\quad (i,j,k)\in T\\
      & z\ \le\  (2b_t - 1) A_t^\pm,\quad t\in T\tikzmark{core-br}\\[20pt]
      & \tikzmark{sub-tl}w_{ij}\ =\ x_i y_j,\quad i,j\in I\tikzmark{sub-br}\\[25pt]
      & \tikzmark{sign-tl}b_t=1,\quad t\in T^+\tikzmark{sign-br}\\
      & b_t=0,\quad t\in T^-\\[20pt]
      & \tikzmark{sym-tl}x_1 = x_5 = y_2 = 0\\
      & x_3 = y_4 = 1\\
      & x_2 \le x_4\\
      & y_1 \le y_5\tikzmark{sym-br}\\
      & x_i\ \le\ x_{i+1},\quad i\ge 6\\[20pt]
      & \tikzmark{dom-tl}A_t^{\pm}\in\left[-\frac{1}{2}, \frac{1}{2}\right],\quad t\in T\\
      & w_{ij}\in[0,1],\quad i,j\in I\\
      & b_t \in \{0,1\}, \quad t\in T,\\
      &  0 \le z \le \Delta_{n-1}\tikzmark{dom-br}
    \end{align*}

    \begin{tikzpicture}[remember picture,overlay]
    \tikzset{
      box/.style={
        rounded corners=2pt,
        inner xsep=6pt,
        inner ysep=4pt,
        line width=0.7pt
      },
      label/.style={
        font=\small\bfseries,
        inner sep=1.5pt,
        rounded corners=2pt
      }
    }

\coordinate (CoreBL) at ($(core-tl |- core-br)+(-100pt,-5pt)$);

\coordinate (CoreTR) at ($(core-br)+(180pt,80pt)$);

\path let \p1 = (CoreBL), \p2 = (CoreTR) in
  node[box, draw=blue!65!black,
       anchor=south west,
       minimum width = {\x2-\x1},
       minimum height= {\y2-\y1}] (Bcore) at (CoreBL) {};

\node[label, text=blue!65!black, fill=none, anchor=south]
  at ($(Bcore.north)$) {Core Model};

\coordinate (SubBL) at ($(sub-tl |- sub-br)+(-5pt,-8pt)$);

\coordinate (SubTR) at ($(CoreTR |- sub-br)+(0pt,15pt)$);

\path let \p1 = (SubBL), \p2 = (SubTR) in
  node[box, draw=teal!65!black,
       anchor=south west,
       minimum width = {\x2-\x1},
       minimum height= {\y2-\y1}] (Bsub) at (SubBL) {};

\node[label, text=teal!65!black, fill=none, anchor=south]
  at (Bsub.north) {Product Substitution};

\coordinate (SignBL) at ($(sign-tl |- sign-br)+(-5pt,-25pt)$);
\coordinate (SignTR) at ($(CoreTR |- sign-br)+(0pt,15pt)$);

\path let \p1 = (SignBL), \p2 = (SignTR) in
  node[box, draw=orange!70!black,
       anchor=south west,
       minimum width = {\x2-\x1},
       minimum height= {\y2-\y1}] (Bsign) at (SignBL) {};

\node[label, text=orange!70!black, fill=none, anchor=south]
  at (Bsign.north) {Sign Fixing};

\coordinate (SymBL) at ($(sym-tl |- sym-br)+(-5pt,-25pt)$);
\coordinate (SymTR) at ($(CoreTR |- sym-br)+(0pt,60pt)$);

\path let \p1 = (SymBL), \p2 = (SymTR) in
  node[box, draw=red!70!black, line width=1.1pt,
       anchor=south west,
       minimum width = {\x2-\x1},
       minimum height= {\y2-\y1}] (Bsym) at (SymBL) {};

\node[label, text=red!70!black, fill=none, anchor=south]
  at (Bsym.north) {Symmetry Breaking};

\coordinate (DomBL) at ($(dom-tl |- dom-br)+(-5pt,-8pt)$);
\coordinate (DomTR) at ($(CoreTR |- dom-br)+(0pt,75pt)$);

\path let \p1 = (DomBL), \p2 = (DomTR) in
  node[box, draw=gray!70!black,
       anchor=south west,
       minimum width = {\x2-\x1},
       minimum height= {\y2-\y1}] (Bdom) at (DomBL) {};

\node[label, text=gray!70!black, fill=none, anchor=south]
  at (Bdom.north) {Variable Domains};
    
    \end{tikzpicture}
    
    \caption{The final MINLP formulation~$P_\Delta^\star$. Colored boxes indicate the core model~$P_\Delta^0$ (blue) and the enhancements: product substitution (teal), sign fixing (orange), symmetry breaking (red), and variable domains (gray).}
    \label{fig:final-model}
    \end{figure}

\subsection{Variants for the types II and III}\label{sec:type-models}

By Theorem~\ref{thm:case-split}, $\Delta_n$ is the largest of the optimal values of three restricted problems. Besides $P_\Delta^\star$ for type~I, we therefore solve two variants that differ from $P_\Delta^\star$ only in the fixings.
\begin{itemize}
  \item \emph{Type II:} $p_1=(0,y_1)$, $p_2=(x_2,0)$, $p_3=C=(1,1)$, $p_4=D=(0,1)$, and $x_5\le\dots\le x_n$. The points $p_1,\dots,p_4$ occur in counterclockwise order, so all triangles among them are positively oriented. In addition, $A^\pm_{34k}=\tfrac12(1-y_k)\ge 0$ and $A^\pm_{14k}=-\tfrac12(1-y_1)x_k\le 0$ for $k\ge 5$.
  \item \emph{Type III:} $p_1=A=(0,0)$, $p_2=C$, $p_3=D$, and $x_4\le\dots\le x_n$, with $A^\pm_{123}>0$, $A^\pm_{23k}=\tfrac12(1-y_k)\ge 0$, and $A^\pm_{13k}=-\tfrac12x_k\le 0$ for $k\ge 4$.
\end{itemize}
The corresponding binary variables are fixed as in Section~\ref{sec: sign fixing}. All other ingredients of $P_\Delta^\star$ are kept.

\subsection{Comparison with Monji, Modir, and Kocuk}\label{sec:comparison-monji}

Concurrent work by Monji, Modir, and Kocuk~\cite{monji2025solvingheilbronntriangleproblem} was the first to apply general-purpose MINLP solvers to the Heilbronn problem and solved $n=9$ to global optimality within solver tolerances. Our formulation differs in several respects:

\begin{enumerate}
  \item \textbf{Absolute value handling:} Monji et al.\ linearize $z\le|A_t^\pm|$ with binary variables and big-$M$ type coefficients. We use the bilinear constraint $z \le (2b_t-1)A_t^\pm$, which avoids such coefficients and is handled by the solver's spatial branching.

  \item \textbf{Symmetry breaking:} We fix five boundary points on specific edges in counterclockwise order, eliminating five coordinates outright, whereas Monji et al.\ order the points by their $y$-coordinates. The fixed boundary order also enables sign fixing for $\binom{5}{3} + (n-5)$ binary variables.

  \item \textbf{Minimal auxiliary variables:} We introduce no edge-proximity, strip-assignment, or grid-cell auxiliary binary variables.

  \item \textbf{No precomputations:} The structural information that Monji et al.\ obtain in about 23 hours of precomputation follows here from Theorem~\ref{thm:case-split}, which is proved analytically (Section~\ref{sec:comparison-effort}).
\end{enumerate}

Section~\ref{sec:comparison-effort} compares the total computational effort of both approaches.

\section{Computational study}\label{sec:computations}

\subsection{Setup}\label{sec:setup}

All computations were carried out on an Apple MacBook Pro with an Apple M5 Pro chip (15-core CPU) and 48\,GB of RAM.
The models are implemented in Python~3.12 using Gurobi~13.0.3 through its Python interface \texttt{gurobipy}. The symbolic computations use SymPy~1.14.0.
The only non-default Gurobi parameter is \texttt{MIPFocus\,=\,2}, which shifts the emphasis towards the dual bound.
All other parameters keep their default values. In particular, the relative optimality gap is $10^{-4}$, the absolute gap $10^{-10}$, the feasibility and optimality tolerances $10^{-6}$, and the integrality tolerance $10^{-5}$. Gurobi used all 15 cores of the machine.
No initial solution is provided, except in the warm-start experiments of Section~\ref{sec:warmstart}.
Since running times of MINLP solvers vary with the random seed, every instance of the main experiment is solved with five seeds, and we report the median and the range.
The ablation and comparison runs use a single seed and a time limit of four hours. The baseline for $n=7$ has a time limit of one hour, and the runs for $n=10$ a time limit of 24 hours.
All tables in this section are generated by a script from the solver output. The code, the logs, and the solution files are part of the accompanying repository.

The tolerances determine what the numbers mean.
A configuration returned by Gurobi satisfies the constraints up to the feasibility tolerance, so its objective value can deviate from the minimum triangle area of its coordinates by about $10^{-6}$, and the reported objective value can even exceed the true optimum. For $n=8$, for example, Gurobi reports $0.07237703$, while the exact optimal value is $(\sqrt{13}-1)/36\approx 0.07237642$.
Coordinates are therefore reliable to about six decimal places, and the dual bound determines the optimal value only up to the relative gap of $10^{-4}$ and up to floating-point effects. This is why we use the solver output as numerical evidence and as input to the exact computations of Section~\ref{sec:framework}, and not as a proof.

\subsection{Main results}\label{sec:main-results}

Table~\ref{tab:main} reports the results of the final model $P^\star_\Delta$ for $n=5,\dots,9$.

\begin{table}[ht]
  \centering
  {\small
\begin{tabular}{r r r r r r r r}
\toprule
$n$ & median [s] & range [s] & nodes & objective & dual bound & rel.\ gap & max viol. \\
\midrule
5 & 0.0 & 0.0--0.0 & 51 & 0.1924501 & 0.1924529 & 1.5e-05 & 3.5e-09 \\
6 & 0.2 & 0.2--0.7 & 13,977 & 0.1250009 & 0.1250058 & 3.9e-05 & 9.3e-07 \\
7 & 0.6 & 0.5--0.8 & 26,091 & 0.0838596 & 0.0838596 & 0.0e+00 & 7.8e-07 \\
8 & 5.9 & 4.9--8.1 & 223,725 & 0.0723770 & 0.0723770 & 0.0e+00 & 8.2e-07 \\
9 & 496.7 & 393.3--752.6 & 18,326,438 & 0.0548768 & 0.0548816 & 8.8e-05 & 6.7e-07 \\
\bottomrule
\end{tabular}
}
  \caption{Results of $P^\star_\Delta$ over five seeds: median running time and range, median number of branch-and-bound nodes, best objective value, dual bound, relative gap, and maximal constraint violation of the best solution.}
  \label{tab:main}
\end{table}

All instances up to $n=8$ are solved in less than ten seconds. For $n=9$, the median running time is 497 seconds, about eight minutes, and the slowest of the five seeds needs about twice as long as the fastest.
Gurobi finds a configuration with the optimal value within the first 15 seconds in every seed, so almost all of the running time is spent on proving the dual bound.
The objective values exceed the exact values of Section~\ref{sec:solutions} by less than $10^{-6}$, as explained in Section~\ref{sec:setup}. The dual bounds exceed the exact values by at most $1.3\cdot 10^{-5}$, or $1.1\cdot 10^{-4}$ in relative terms. This corresponds to the relative gap of $10^{-4}$, measured against an objective value that the feasibility tolerance slightly inflates.

\subsection{Configurations with adjacent corners}\label{sec:types-results}

Table~\ref{tab:cases} reports the dual bounds of the variants for the types~II and~III of Theorem~\ref{thm:case-split}.

\begin{table}[ht]
  \centering
  {\small
\begin{tabular}{r r r r r}
\toprule
$n$ & bound II & time II [s] & bound III & time III [s] \\
\midrule
5 & 0.1666812 & 0.0 & 0.1464467 & 0.0 \\
6 & 0.1250000 & 0.0 & 0.1250001 & 0.0 \\
7 & 0.0838616 & 0.2 & 0.0773507 & 0.1 \\
8 & 0.0669875 & 1.5 & 0.0669875 & 0.9 \\
9 & 0.0522902 & 69.0 & 0.0502740 & 23.8 \\
10 & 0.0418335 & 5,463 & 0.0418329 & 292 \\
\bottomrule
\end{tabular}
}
  \caption{Dual bounds and running times of the variants for the types~II and~III.}
  \label{tab:cases}
\end{table}

For every $n\le 10$, the bounds of types~II and~III do not exceed the value of type~I within the tolerances, so the optimal value is attained by a configuration of type~I. For $n\ge 8$ the bounds are clearly smaller, for $n=9$ by 4.7\% for type~II and by 8.4\% for type~III, and for $n=10$ by 10.1\% for both.
For $n=6$, all three types attain the value $1/8$. For $n=7$, type~II attains the same value as type~I, as expected, since the optimal configuration has both types (Section~\ref{sec: boundary configuration}).
For $n\le 9$, the variants for types~II and~III are solved in at most 69 seconds. For $n=10$, type~II takes 1.5 hours and type~III five minutes. Their running times are part of the total effort in Section~\ref{sec:comparison-effort}.

\subsection{Contribution of the model ingredients}\label{sec:ablation}

To quantify the contribution of each ingredient of $P^\star_\Delta$, we remove one ingredient at a time and solve the resulting model for $n=7,8,9$.
The variants are: no boundary pattern (no fixed coordinates and hence no sign fixing, all points ordered by $x$-coordinate), no sign fixing, no ordering of the interior points, no product substitution, no bound $z\le\Delta_{n-1}$, and binary variables only for triangles whose sign is not fixed (Remark in Section~\ref{sec: sign fixing}).
In addition, we test the strip capacities described below.
Table~\ref{tab:ablation} also contains the baseline $P^0_\Delta$, the warm-start runs of Section~\ref{sec:warmstart}, and the runs with Gurobi~11.0.1 of Section~\ref{sec:comparison-effort}.

\begin{table}[ht]
  \centering
  {\small
\begin{tabular}{l r r r}
\toprule
variant & $n=7$ [s] & $n=8$ [s] & $n=9$ [s] \\
\midrule
full $P^\star_\Delta$ (seed 0) & 0.5 & 5.9 & 393 \\
no boundary pattern & 21.5 & 262 & $>$14,400 \\
no sign fixing & 0.4 & 6.3 & 774 \\
no interior order & 0.6 & 19.9 & $>$14,400 \\
no substitution & 0.6 & 7.7 & 387 \\
no $z\le\Delta_{n-1}$ & 0.3 & 5.8 & 862 \\
fixed signs w/o binaries & 0.6 & 4.4 & 338 \\
with strip capacities & 0.8 & 5.3 & 578 \\
warm start & 0.3 & 6.5 & 572 \\
Gurobi 11.0.1 & 0.4 & 5.8 & 365 \\
baseline $P^0_\Delta$ & $>$3,600 & -- & -- \\
\bottomrule
\end{tabular}
}
  \caption{Running times in seconds of variants of $P^\star_\Delta$ (seed 0). Entries marked with $>$ reached the time limit of four hours (one hour for the baseline).}
  \label{tab:ablation}
\end{table}

The ablation uses a single seed. Since the running times of the full model for $n=9$ vary between 393 and 753 seconds over the seeds, only differences beyond this range are meaningful.
The boundary pattern has by far the largest effect. Without it, the model is about 40 times slower for $n=7$ and $n=8$, and for $n=9$ the relative gap is still 17\% after four hours.
The ordering of the interior points comes second. Without it, $n=8$ takes 3.4 times as long, and $n=9$ ends after four hours with a relative gap of 8\%.
Without the sign fixing and without the bound $z\le\Delta_{n-1}$, the running time for $n=9$ roughly doubles (774 and 862 seconds), while $n=7$ and $n=8$ are unaffected.
The product substitution and the restriction of the binary variables to the triangles whose sign is not fixed have no measurable effect. A likely reason for the former is that Gurobi introduces auxiliary variables for bilinear terms itself.

\paragraph{Strip capacities.}
Monji et al.~\cite[Table~1]{monji2025solvingheilbronntriangleproblem} bound the sorted $y$-coordinates by the number of points that fit into a horizontal strip of height $1/\kappa$ without creating a triangle of area below a lower bound $H\le\Delta_n$, and they determine these capacities by solving restricted Heilbronn problems.
The capacities also follow without computation.
An affine scaling maps a strip of height $t$ onto the unit square and multiplies all areas by $1/t$. If $m+1$ points with minimum triangle area at least $H$ fit into the strip, their images have minimum area at least $H/t$, which is impossible if $H/t>\Delta_{m+1}$.
Hence every strip of height $t<H/\Delta_{m+1}$ along any edge of the square contains at most $m$ points of a configuration with value at least $H$.
With the proven values $\Delta_3,\dots,\Delta_8$ this reproduces all bounds of~\cite[Table~1]{monji2025solvingheilbronntriangleproblem} for $n=7,8,9$, and choosing $t$ just below $H/\Delta_{m+1}$ instead of $1/\kappa$ gives slightly stronger ones. For $n=9$, for example, at most two points lie below $y=0.1097$, compared with $y\ge 1/10$ for the third point in~\cite{monji2025solvingheilbronntriangleproblem}.
The resulting constraints are valid for every configuration with value at least $H$, independently of labeling and symmetry breaking, so they can be added to $P^\star_\Delta$. We model them with one binary variable per point, strip, and edge.
The strip capacities do not speed up the model (Table~\ref{tab:ablation}). For $n=9$, the running time of 578 seconds lies within the range of the seeds of the full model. The boundary pattern and the ordering of the interior points apparently already contain most of the information that the capacities provide.

\subsection{Initialization}\label{sec:warmstart}

By default, Gurobi starts without an initial solution and finds feasible configurations with its own heuristics. To assess the effect of a good starting point, we also provide the best-known configuration as a MIP start.
The warm start does not shorten the solution process (Table~\ref{tab:ablation}). For $n=9$, it takes 572 seconds, which lies within the range of the seeds without a warm start. This matches the observation of Section~\ref{sec:main-results} that Gurobi finds the optimal configuration within seconds and spends almost all of the time on the dual bound.

\subsection{Comparison of the total effort}\label{sec:comparison-effort}

Table~\ref{tab:effort} compares the total computational effort for $n=9$ with that of Monji, Modir, and Kocuk~\cite{monji2025solvingheilbronntriangleproblem}.
For their approach, the total effort consists of the two runs that establish the near-boundary property (318 and 82{,}359 seconds), the computation of the variable bounds (35 seconds), and the final solve (633 seconds), about 23 hours in total.
For our approach, it consists of the solves for the three types of Theorem~\ref{thm:case-split}.
The two computations used different hardware and different versions of Gurobi. Their workstation had 96 cores, ours 15. To separate the effect of the solver version from that of the formulation, we also solved our model with Gurobi~11.0.1, the version used in~\cite{monji2025solvingheilbronntriangleproblem}, on our machine.

\begin{table}[ht]
  \centering
  {\small
  \begin{tabular}{l l l r}
    \toprule
    & solver & hardware & total time [s] \\
    \midrule
    Monji et al.~\cite{monji2025solvingheilbronntriangleproblem} & Gurobi 11.0.1 & 2 $\times$ 48-core Xeon Gold 6248R & $318+82{,}359+35+633$ \\
    this paper & Gurobi 13.0.3 & Apple M5 Pro (15 cores) & $393+69+24$ \\
    this paper & Gurobi 11.0.1 & Apple M5 Pro (15 cores) & $365$ (type~I) \\
    \bottomrule
  \end{tabular}}
  \caption{Total computational effort for $n=9$.}
  \label{tab:effort}
\end{table}

With seed 0, the three types of Theorem~\ref{thm:case-split} take 486 seconds in total, and with the slowest seed for type~I 846 seconds, compared with about 83{,}300 seconds for Monji et al.
Under Gurobi~11.0.1, the type~I model takes 365 seconds, slightly less than with seed 0 under Gurobi~13.0.3 (393 seconds), so the solver version does not explain the difference.
The final solves are of the same order (633 seconds on 96 cores against 393 to 753 seconds on 15 cores). The difference in the total effort stems from the precomputation, which our approach does not need, since the structural information follows from Theorem~\ref{thm:case-split}.

\subsection{The case $n=10$}\label{sec:n10}

For $n=10$, we ran $P^\star_\Delta$ and the variants for types~II and~III with a time limit of 24 hours each, with the cap $z\le 0.0548818$ of Section~\ref{sec:bounds-subst}. No previous computation had attempted this case without structural assumptions. All three runs terminated well before the limit with Gurobi's default optimality tolerance.

\begin{itemize}
  \item Type~I took 18{,}039 seconds (5.0 hours) and 446 million nodes. It ended with the objective value $0.0465383$ and the dual bound $0.0465430$, a relative gap of $1.0\cdot 10^{-4}$.
  \item Types~II and~III took 5{,}463 and 292 seconds, with dual bounds $0.0418335$ and $0.0418329$ (Table~\ref{tab:cases}). These lie 10.1\% below the type~I value, so no configuration of these types comes close to the optimum.
\end{itemize}

\noindent
Together, the three runs give $\Delta_{10}\le 0.0465430$ numerically, after 6.6 hours in total.

The best-known configuration of Comellas and Yebra~\cite{comellasyebra2002} has the exact value $0.0465374$ (Section~\ref{sec:solutions}), $1.2\cdot 10^{-4}$ relative below the dual bound, the same order as for $n\le 9$ (Section~\ref{sec:main-results}). The second solution in Gurobi's solution pool coincides with it up to a symmetry of the square, with a deviation below $10^{-6}$ per coordinate. The final incumbent lies within $10^{-3}$ of it, and its reported objective value exceeds the exact value only because it uses the feasibility tolerance: its true minimum triangle area is $0.0465365$. Hence the configuration of Comellas and Yebra is optimal for $n=10$ within solver tolerances, with the same status as $n=9$ plus the dependence on the numerical bound for $\Delta_9$ through the cap.

Section~\ref{sec:solutions} derives the exact coordinates of this configuration from the final incumbent with the framework of Section~\ref{sec:framework} and verifies them in exact arithmetic, as for $n\le 9$. This gives the rigorous lower bound $\Delta_{10}\ge 0.0465374$, so the numerical upper bound encloses $\Delta_{10}$ in an interval of relative width $1.2\cdot 10^{-4}$.

Monji et al.~\cite{monji2025solvingheilbronntriangleproblem} report for $n=10$ a lower bound of $0.0465383$ (computed area $0.0465369$) after about 800 seconds, under two unproven structural assumptions (exactly two points on each edge, and $y_5\le 1/2$ for the sorted $y$-coordinates), and no upper bound within one day. They describe the case as open. Our model needs no such assumption, since Theorem~\ref{thm:case-split} covers all configurations.

\section{The optimize-then-refine framework}\label{sec:framework}

  \subsection{The framework}\label{sec:framework-sub}

Our approach proceeds in two stages.
In the first step, we solve the MINLP formulation~$P_\Delta^\star$ from Section~\ref{sec:final-model} with Gurobi to global optimality within the solver tolerances.
The solver returns a numerical configuration together with a dual bound that agrees with its objective value up to the relative gap of $10^{-4}$. Both are computed in floating-point arithmetic (Section~\ref{sec:setup}).
From the numerical coordinates we identify the \emph{critical triangles}, that is, the triangles whose area equals the minimum up to the solver tolerance, and record which points lie on which edges of the unit square.
In the second step, we treat these structural observations as an ansatz for deriving exact symbolic coordinates:
setting the areas of the critical triangles equal yields a system of polynomial equations that can be solved in closed form with a computer algebra system.
Finally, we verify the resulting configuration in exact arithmetic: all $\binom{n}{3}$ triangle areas are computed exactly, and we check that none of them is smaller than the claimed minimum and that exactly the expected triangles attain it.

The three results of the pipeline have a different mathematical status, and we keep them apart.
The exact verification makes the reconstructed configuration a \emph{rigorous lower bound} on $\Delta_n$.
It does not prove optimality. The structure used in Step~2 is inferred from the numerical solution, so the solution of the polynomial system is exact only for this structure, and a configuration with a different set of critical triangles is not excluded by Step~2.
The global upper bound comes from Step~1 and is numerical, so it is not a mathematical proof~\cite{neumaier2004complete}.
For $n\le 8$, optimality is proven in the literature (Section~\ref{sec:intro-small-n}), and our computations confirm these results. For $n=9$ and $n=10$, the exact configuration matches the solver's numerical upper bound, but no rigorous proof of optimality is known.
\begin{remark}
  A key practical concern is the identification of critical triangles and boundary points in the first step, since these decisions rest on numerical data subject to solver tolerances.
In all instances considered here, the gaps are unambiguous:
the area of the smallest non-critical triangle exceeds that of the critical triangles by a margin that is several orders of magnitude larger than the solver tolerance, and every point classified as lying on an edge has a coordinate that deviates from the boundary value by less than~$10^{-6}$. 
This clear separation makes the structural input to Step~2 unambiguous in all cases considered.
\end{remark}

The pipeline is illustrated in Figure~\ref{fig:pipeline}.

\begin{figure}[h]
  \centering
  \begin{tikzpicture}[
    font=\small,
    >=Latex,
    node distance=10mm,
    box/.style={
      draw,
      rounded corners=3mm,
      inner sep=7pt,
      text width=0.60\linewidth, 
      align=left
    },
    boxA/.style={box, fill=blue!6,  draw=blue!60!black},
    boxB/.style={box, fill=gray!8,  draw=gray!70!black},
    boxC/.style={box, fill=green!6, draw=green!60!black},
    arrow/.style={->, line width=0.9pt},
    dashedarrow/.style={->, line width=0.8pt, dashed}
  ]

  \newcommand{\brow}[1]{\textbullet\; #1\\[-1pt]}
  
  \node[boxA] (s1) {
    \textbf{Step 1: Numerical global optimization}\par\vspace{2pt}
    \begin{tabular}{@{}p{\linewidth}@{}}
      \brow{Create MINLP model}
      \brow{Compute global optimal point numerically}
      \brow{Identify critical triangles and boundary points}
    \end{tabular}
  };
  
  \node[boxB, below=of s1] (s2) {
    \textbf{Step 2: Exact analytical refinement}\par\vspace{2pt}
    \begin{tabular}{@{}p{\linewidth}@{}}
      \brow{Derivation of nonlinear polynomial equation systems}
      \brow{Exact symbolic solution of the systems}
      \brow{Exact verification of all triangle areas}
    \end{tabular}
  };

  \node[boxC, below=of s2] (result) {
    \textbf{Result:} Exact coordinates with an exactly verified minimum area (rigorous lower bound) that matches the solver's numerical upper bound
  };

  \draw[arrow] (s1.south) -- node[right, font=\footnotesize] {numerical configuration} (s2.north);
  \draw[arrow] (s2.south) -- (result.north);
  
  \end{tikzpicture}
  
\caption{Methodological pipeline of the paper: numerical global optimization (Step~1) and exact analytical refinement with exact verification (Step~2). The result is a rigorous lower bound, while the upper bound remains numerical.}
  \label{fig:pipeline}
  \end{figure}

Given the formulation of Section~\ref{sec:MIPModel}, Step~1 is a routine solver call. Step~2 is more involved.
The equal-area conditions on the critical triangles form a system of polynomial equations whose unknowns are the free coordinate parameters remaining after boundary constraints and coordinate coincidences have been exploited.
The number~$k$ of these parameters ranges from~$3$ (for $n=10$) to~$8$ (for $n=9$), and we employ three solving strategies, tried in succession, and for $n=10$ an additional stationarity condition.

\begin{enumerate}
\item \emph{Direct symbolic solve.}
For small systems ($k \le 6$, covering $n=5,6,8$), SymPy's symbolic solver produces exact solutions directly, including the detection of the one-parameter family for $n=6$.
\item \emph{Gröbner basis.}
A Gröbner basis of a polynomial ideal with respect to a monomial order is a generating set with particularly good division properties~\cite{cox2015ideals}. For a lexicographic order it is triangular: the last polynomials involve only the last variables, so the system can be solved by back-substitution, much like Gaussian elimination for linear systems.
For $n=7$, the solution lies in a cubic extension field, and a lexicographic Gröbner basis computation
yields a triangular system from which all coordinates follow by back-substitution through the minimal polynomial $19f^3 - 27f^2 + 11f - 1 = 0$ (see Section~\ref{sec:illustration-n7}).
\item \emph{Numeric-to-symbolic refinement.}
For $n=9$, where $k=8$, neither a direct solve nor a Gröbner basis computation terminates in reasonable time.
Instead, we apply a numeric-to-symbolic approach: from the numerical coordinates we first identify the underlying number field~$\mathbb{Q}(\sqrt{65})$ by recognizing~$\sqrt{65}$ in the simplest coordinate, then express all coordinates as elements of this field using SymPy's \texttt{nsimplify} function with the known algebraic extension.
\end{enumerate}

\noindent
The numeric-to-symbolic step involves a risk:
numerical proximity does not guarantee algebraic correctness, since the coordinates returned by Gurobi are reliable only to about six decimal places (Section~\ref{sec:setup}).
For instance, for $n=9$ the spurious expression $\tfrac{19}{166} + \tfrac{3\sqrt{65}}{166} \approx 0.26017$
is numerically indistinguishable from the true value $\tfrac{9}{16} - \tfrac{3\sqrt{65}}{80} \approx 0.26017$
at solver precision, yet only the latter satisfies the equal-area equations.
We therefore verify every candidate solution symbolically by substituting it into the polynomial system of the critical triangles.
Any coordinate that fails this check is corrected by fixing the validated coordinates and solving the resulting reduced (and much smaller) subsystem exactly.

For $n=10$, the equal-area conditions alone do not determine the configuration. With the coordinate coincidences read off from the numerical solution, $k=3$ parameters remain, and the 16 critical areas take only three distinct forms $A_1$, $A_2$, $A_3$. The equations $A_1=A_2=A_3$ therefore leave a curve of solutions. Since the configuration maximizes the common area along this curve, we add the Lagrange condition $\det(\nabla A_1,\nabla A_2,\nabla A_3)=0$. A lexicographic Gröbner basis of the resulting system, after the removal of degenerate solutions in which a parameter vanishes, contains the polynomial $(2g-1)(12g^3-27g^2+20g-4)$ in the coordinate $g$ of the point $p_7=(g,g)$. The root $g=1/2$ is spurious, and the other two parameters follow as polynomials in the real root of the cubic factor. This reproduces the configuration of Comellas and Yebra~\cite{comellasyebra2002}, who also obtained it by an analytical refinement.

The final exact verification is reported in Section~\ref{sec:solutions}.
Step~2 takes less than one second for every $n\le 10$, and the exact verification less than seven seconds per configuration.

The entire pipeline is implemented as a reproducible Python script, available together with all input data in the accompanying repository.\footnote{\url{https://github.com/spiralulam/heilbronn}}

  \subsection{Illustration of the methodology for $n=7$}
  \label{sec:illustration-n7}

  We now illustrate the two-step pipeline on the case $n=7$, for which
  Comellas and Yebra~\cite{comellasyebra2002} heuristically obtained conjectured optimal
  coordinates via simulated annealing followed by a local optimality analysis.
  Global optimality was subsequently proved by Zeng and Chen~\cite{ZengChen2011}, who decomposed the problem into 226 nonlinear optimization subproblems.
  Our computation independently confirms their result numerically and yields simpler coordinate expressions.

  \subsubsection*{Step 1: Numerical global optimization}

  Running our final MINLP model from Section~\ref{sec:final-model} with
  $n=7$, Gurobi terminates after 0.8 seconds (seed 1) with an objective value of $0.0838594$, close to $\Delta_7 \approx 0.0838590$, and
  the numerical configuration shown in
  Table~\ref{tab:n7-numerical}.

  \begin{figure}[h]
    \centering
    \begin{minipage}[c]{0.46\linewidth}
      \centering
      \small
      \begin{tabular}{c c c c}
      \toprule
      Pt & $x$ & $y$ & Edge \\
      \midrule
      $p_1$ & $0$      & $0.1352$ & left edge \\
      $p_2$ & $0.7127$ & $0$      & bottom edge \\
      $p_3$ & $1$      & $0.1808$ & right edge \\
      $p_4$ & $1$      & $1$      & corner \\
      $p_5$ & $0$      & $1$      & corner \\
      $p_6$ & $0.1939$ & $0.4926$ & interior \\
      $p_7$ & $0.7127$ & $0.5839$ & shared $x$ with $p_2$ \\
      \bottomrule
      \end{tabular}
      \captionof{table}{Numerical coordinates returned by Gurobi for $n=7$
       (rounded to four decimal places). Five of the seven points lie on the
       boundary of the unit square.}
      \label{tab:n7-numerical}
    \end{minipage}\hfill
    \begin{minipage}[c]{0.5\linewidth}
      \centering
      \includegraphics[width=\linewidth,trim=35pt 10pt 35pt 10pt,clip]{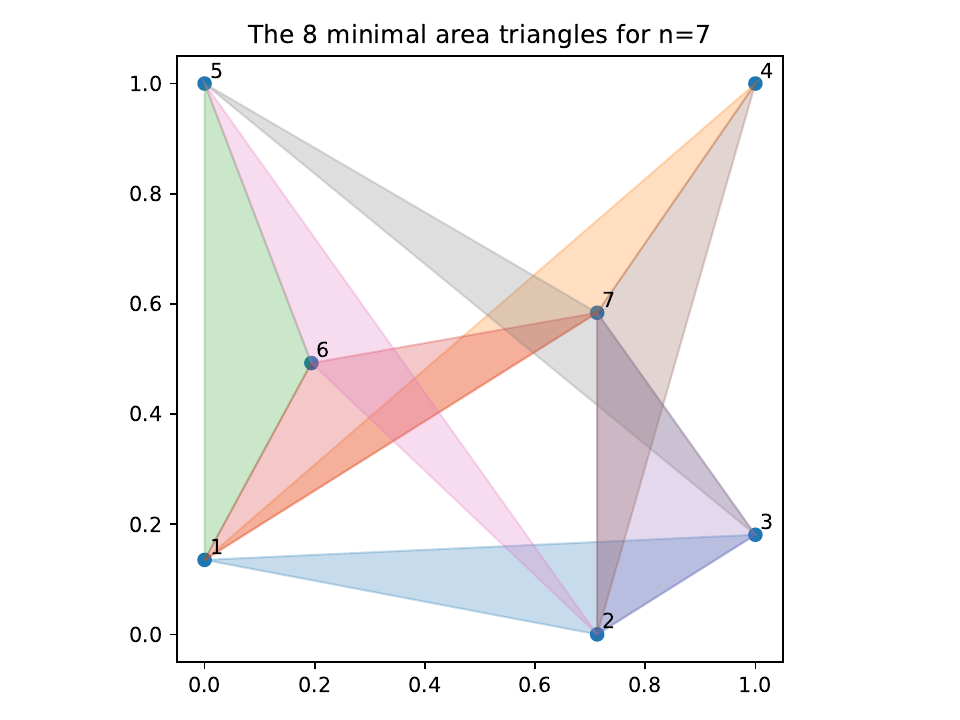}
      \captionof{figure}{Optimal configuration for $n=7$
        with its eight critical triangles highlighted.}
      \label{fig:n7-config-step1}
    \end{minipage}
  \end{figure}

  Figure~\ref{fig:n7-config-step1} shows the configuration
  together with its eight critical triangles, which all have areas within
  $10^{-6}$ of each other. The ninth-smallest
  triangle area is roughly $0.0948$, leaving a comfortable gap
  of about $13\%$ above the optimum.

  Two structural observations guide the symbolic step that follows:
  \begin{enumerate}
    \item Five of the seven points lie on the boundary
          ($p_4, p_5$ at corners, $p_1$ on the left edge,
           $p_2$ on the bottom edge, $p_3$ on the right edge).
    \item Points $p_2$ and $p_7$ share the same $x$-coordinate,
          suggesting a hidden algebraic relation.
  \end{enumerate}

  \subsubsection*{Step 2: Exact analytical refinement}

  Guided by these observations, we set up a parametric ansatz.
  The boundary structure fixes seven coordinates
  ($x_1 = 0$, $y_2 = 0$, $x_3 = 1$, $x_4 = y_4 = 1$, $x_5 = 0$, $y_5 = 1$),
  and the shared $x$-coordinate of $p_2$ and $p_7$ introduces an additional coupling,
  leaving six unknowns $a, b, c, d, e, f$ (Table~\ref{tab:n7-parametric}).

  \begin{table}[h]
    \centering
    \begin{tabular}{c c c l}
    \toprule
    Pt & $x$ & $y$ & Rationale \\
    \midrule
    $p_1$ & $0$ & $a$ & left edge \\
    $p_2$ & $b$ & $0$ & bottom edge \\
    $p_3$ & $1$ & $c$ & right edge \\
    $p_4$ & $1$ & $1$ & corner \\
    $p_5$ & $0$ & $1$ & corner \\
    $p_6$ & $d$ & $e$ & interior \\
    $p_7$ & $b$ & $f$ & shared $x$ with $p_2$ \\
    \bottomrule
    \end{tabular}
    \caption{Parametric ansatz for $n=7$. Boundary constraints
    fix seven coordinates, and the shared $x$-coordinate of $p_2$
    and $p_7$ introduces the coupling $x_7 = b$.}
    \label{tab:n7-parametric}
  \end{table}

  Setting the areas of the eight critical triangles equal yields
  seven polynomial equations in six unknowns.
  A lexicographic Gröbner basis computation eliminates all unknowns
  except~$f$, which must satisfy the irreducible cubic
  \[
    19\,f^{3} - 27\,f^{2} + 11\,f - 1 \;=\; 0.
  \]
  This cubic has three real roots:
  $f_1 \approx 0.1269$,
  $f_2 \approx 0.5839$, and
  $f_3 \approx 0.7103$.
  The first root gives $b = 19f_1^2 - 27f_1 + 10 \approx 6.88 > 1$,
  and the third gives $a = 19f_3^2 - 16f_3 + 3 \approx 1.22 > 1$;
  both violate the unit-square constraint.
  The only solution in the unit square therefore corresponds to $f = f_2$.
  Back-substitution through the Gröbner basis yields the
  exact symbolic coordinates in Table~\ref{tab:n7-exact}.

  \begin{table}[h]
    \centering
    \begin{tabular}{c c c}
    \toprule
    Pt & $x$ & $y$ \\
    \midrule
    $p_1$ & $0$ & $19f^{2} - 16f + 3$ \\[6pt]
    $p_2$ & $19f^{2} - 27f + 10$ & $0$ \\[6pt]
    $p_3$ & $1$ & $\displaystyle \frac{-19f^{2} + 10f + 1}{2}$ \\[10pt]
    $p_4$ & $1$ & $1$ \\[6pt]
    $p_5$ & $0$ & $1$ \\[6pt]
    $p_6$ & $-19f^{2} + 8f + 2$ & $57f^{2}-41f+5$ \\[6pt]
    $p_7$ & $19f^{2} - 27f + 10$ & $f$ \\[6pt]
    \bottomrule
    \end{tabular}
    \caption{Exact symbolic coordinates of the optimal seven-point
    configuration, where $f \approx 0.5839$ is the second real root of
    $19f^3-27f^2+11f-1$.}
    \label{tab:n7-exact}
  \end{table}

  The optimal value simplifies to
  \[
    \Delta_7 \;=\; f - \tfrac{1}{2}
              \;\approx\; 0.0838590090\ldots
  \]

  \paragraph{Comparison with Comellas and Yebra~\cite{comellasyebra2002}.}
  The exact coordinates we obtain are notably simpler than those
  reported by Comellas and Yebra, who express
  each coordinate as a degree-two polynomial in a parameter~$z$
  with rational coefficients over denominators~$19$ and~$38$
  and integer numerators as large as~$223$, for
  instance $x_1 = -\tfrac{50}{19}\,z - \tfrac{17}{38}\,z^{2}
  + \tfrac{37}{38}$.
  
\section{Exact configurations for $5\le n\le 10$}\label{sec:solutions}

This section reports the exact configurations obtained with the framework of Section~\ref{sec:framework}.
Exact forms are also known for several best-known configurations with $n\ge 11$, which are not known to be optimal: for $n=11$ by Goldberg~\cite{goldberg1972maximizing}, for $n=12$ by Comellas and Yebra~\cite{comellasyebra2002}, for $n=14$ by Chappell~\cite{ChappellN14}, and for $n=16$ by Beyleveld (Appendix~\ref{app:best-known}). No exact forms are known for the asymmetric best-known configurations for $n=13$, due to Karpov, and $n=15$, found by the author in 2026 with a heuristic search, which replaced Karpov's configuration from 2011. The website of Tej Stead~\cite{SteadHeilbronn} lists high-precision coordinates for all of them and exact values where known.
For every $n=5,\dots,10$, we computed all $\binom{n}{3}$ triangle areas of the configurations below in exact arithmetic over the respective number field. None of them is smaller than the stated value $\Delta(P)$, and exactly the critical triangles attain it. For $n\le 8$, these configurations are proven optimal, so $\Delta(P)=\Delta_n$. For $n=9$ and $n=10$, $\Delta(P)$ is a rigorous lower bound on $\Delta_n$ that matches the solver's numerical upper bound. This verification takes less than seven seconds per configuration. For $n=6$, the members $c=0$ and $c=1/8$ of the one-parameter family were verified.

  Table~\ref{tab:optimal-values} summarizes the values, Figure~\ref{fig:gallery} shows the configurations with their critical triangles, and Appendix~\ref{app:exact-coordinates} lists the exact coordinates.
For $n=5$, $8$, and $9$, the polynomial system defined by the critical triangles and the observed boundary structure has a unique solution in the unit square (this does not imply global uniqueness).
For $n=10$, the configuration is determined by the equal-area conditions together with the stationarity condition of Section~\ref{sec:framework-sub}.
For $n=6$, the optimal configurations form a one-parameter family with $c\in[0,\tfrac14]$, all with the same distribution of triangle areas, which is consistent with the count of active constraints in Section~\ref{sec:critical-triangle-count}.
The derivation for $n=7$ is given in Section~\ref{sec:illustration-n7}.

\begin{table}[ht]
  \centering
  \begin{tabular}{c c c c l}
    \toprule
    $n$ & $\Delta(P)$ & $\Delta(P)\approx$ & critical triangles & symmetry \\
    \midrule
    5 & $\sqrt{3}/9$ & $0.19245$ & 4 & reflection in a diagonal \\
    6 & $1/8$ & $0.12500$ & 6 & rotation by $180^\circ$ (family) \\
    7 & $f-\tfrac12$ (Section~\ref{sec:illustration-n7}) & $0.08386$ & 8 & none \\
    8 & $(\sqrt{13}-1)/36$ & $0.07238$ & 12 & rotation by $180^\circ$ \\
    9 & $(9\sqrt{65}-55)/320$ & $0.05488$ & 11 & reflection in a diagonal \\
    10 & $\tfrac58 g^2-\tfrac12 g^3$ (Appendix~\ref{app:exact-coordinates}) & $0.04654$ & 16 & reflections in both diagonals \\
    \bottomrule
  \end{tabular}
  \caption{Values $\Delta(P)$ of the configurations for $5\le n\le 10$, which equal $\Delta_n$ for $n\le 8$.}
  \label{tab:optimal-values}
\end{table}

\begin{figure}[!htbp]
  \centering
  \begin{minipage}[t]{0.31\linewidth}\centering
    \includegraphics[width=\linewidth,trim=62pt 12pt 72pt 26pt,clip]{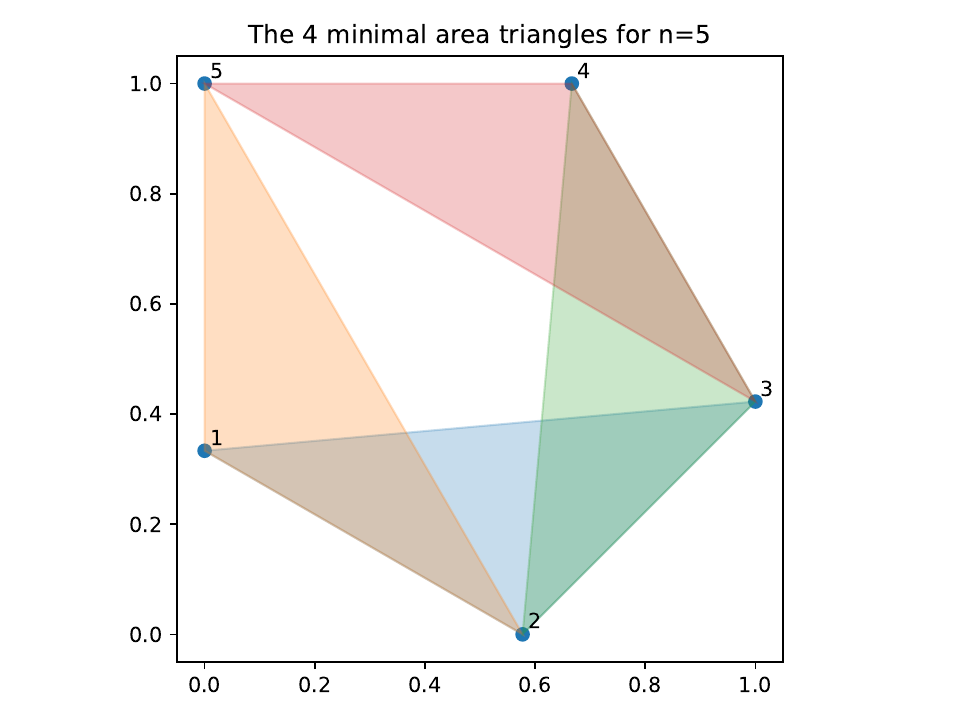}\\[2pt]
    {\small $n=5$, $\Delta(P)=\frac{\sqrt3}{9}$}
  \end{minipage}\hfill
  \begin{minipage}[t]{0.31\linewidth}\centering
    \includegraphics[width=\linewidth,trim=62pt 12pt 72pt 26pt,clip]{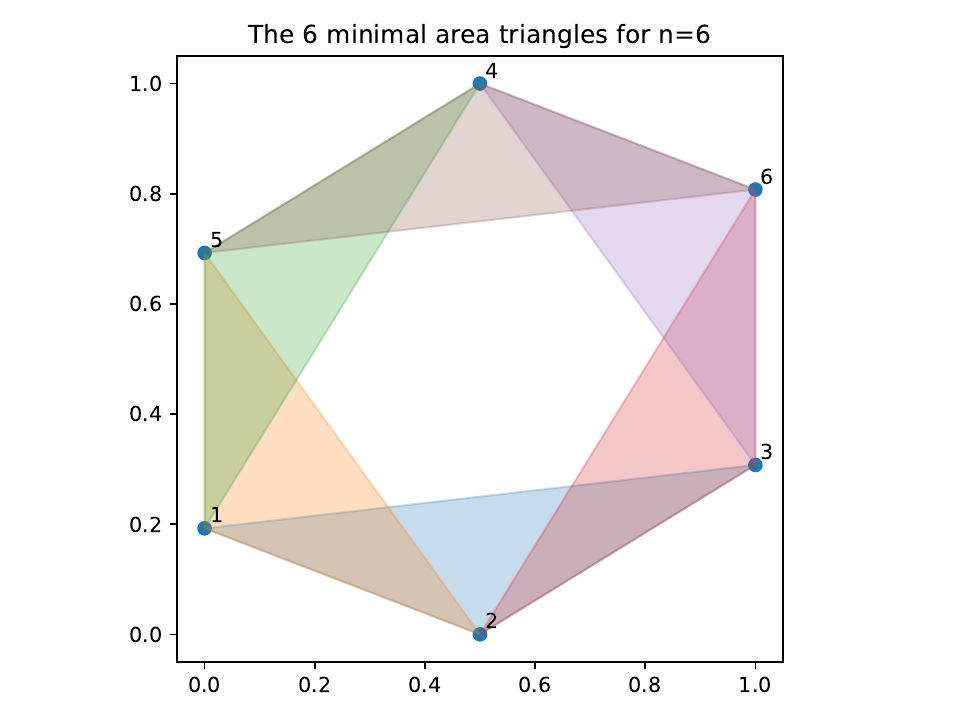}\\[2pt]
    {\small $n=6$, $\Delta(P)=\frac18$}
  \end{minipage}\hfill
  \begin{minipage}[t]{0.31\linewidth}\centering
    \includegraphics[width=\linewidth,trim=62pt 12pt 72pt 26pt,clip]{figs/fig_n7.pdf}\\[2pt]
    {\small $n=7$, $\Delta(P)=f-\frac12$}
  \end{minipage}

  \medskip
  \begin{minipage}[t]{0.31\linewidth}\centering
    \includegraphics[width=\linewidth,trim=62pt 12pt 72pt 26pt,clip]{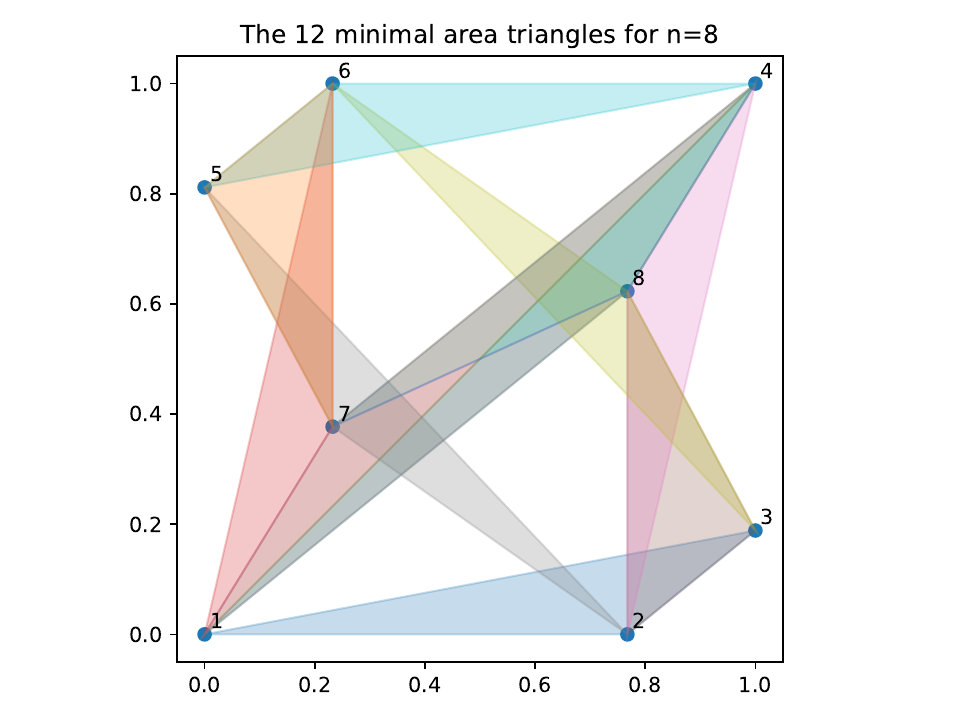}\\[2pt]
    {\small $n=8$, $\Delta(P)=\frac{\sqrt{13}-1}{36}$}
  \end{minipage}\hfill
  \begin{minipage}[t]{0.31\linewidth}\centering
    \includegraphics[width=\linewidth,trim=62pt 12pt 72pt 26pt,clip]{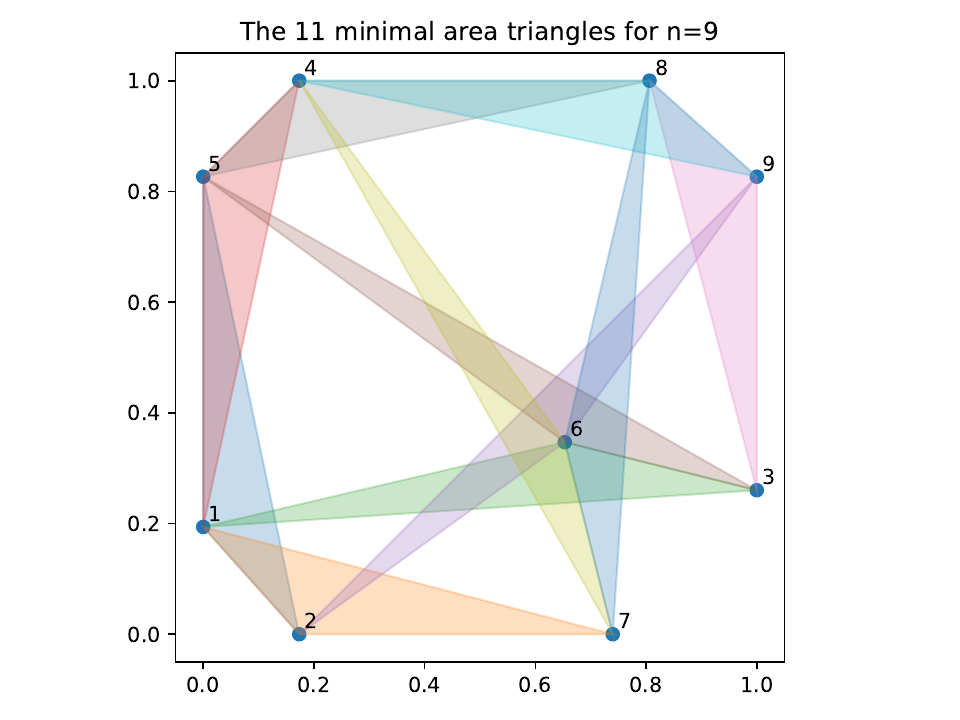}\\[2pt]
    {\small $n=9$, $\Delta(P)=\frac{9\sqrt{65}-55}{320}$}
  \end{minipage}\hfill
  \begin{minipage}[t]{0.31\linewidth}\centering
    \includegraphics[width=\linewidth,trim=62pt 12pt 72pt 26pt,clip]{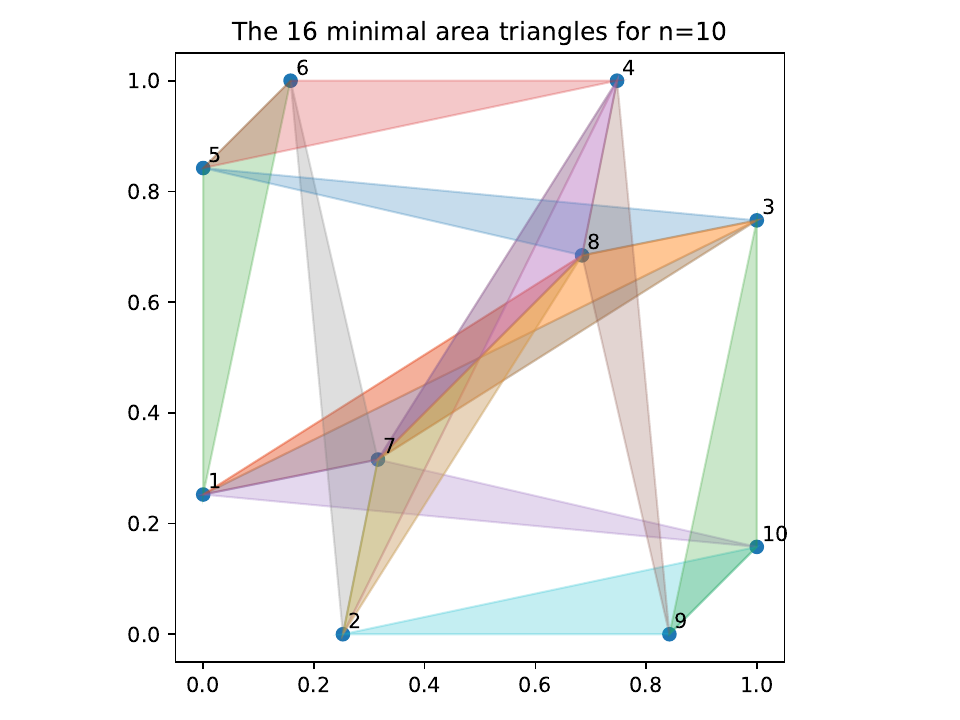}\\[2pt]
    {\small $n=10$, $\Delta(P)=\frac58 g^2-\frac12 g^3$}
  \end{minipage}
  \caption{The configurations for $5\le n\le 10$ with their critical triangles highlighted.}
  \label{fig:gallery}
\end{figure}

\FloatBarrier
\section{Further observations and research questions}
\label{sec:observations}

This section collects structural observations on the configurations for $n\le 10$ and on the best-known configurations for $11\le n\le 16$ (Appendix~\ref{app:best-known}).
For $n\le 10$ they are based on the exact coordinates of Section~\ref{sec:solutions}. For $n\ge 11$ they are based on coordinates with 30 significant digits, and two areas are considered equal if they differ by less than $10^{-12}$.

\subsection{How many triangles are critical?}
\label{sec:critical-triangle-count}

At an optimal configuration, no small motion of the points can enlarge the smallest triangle: every such motion either shrinks a critical triangle or pushes a point out of the square. If the optimum is isolated and nondegenerate, this blocking requires at least as many active constraints as there are degrees of freedom, just as a vertex of a polyhedron in $\mathbb{R}^d$ lies on at least $d$ facets.
To make this precise, note that locally, at a configuration without collinear points, the Heilbronn problem is the smooth problem of maximizing $z$ in the $2n+1$ variables $(x,y,z)$ subject to $z\le \pm A^\pm_t$ for all triangles $t$, with the sign of the current orientation, and to the box constraints.
At an optimal configuration, the active constraints are the critical triangles and the box constraints of the coordinates on $\partial S$, where a corner contributes two. The box constraints cannot do much of the blocking:

\begin{lemma}\label{lem:active-count}
At most eight box constraints are active at a configuration without three collinear points. Consequently, if at least $2n+1$ constraints are active at an optimal configuration, then it has at least $2n-7$ critical triangles.
\end{lemma}

\begin{proof}
Every active box constraint corresponds to a pair of a point and an edge of $S$ that contains it, where a corner belongs to two edges and accounts for two active constraints. No edge contains three points, since they would be collinear, so there are at most $4\cdot 2=8$ such pairs. The second statement follows, since the critical triangles are the only other active constraints.
\end{proof}

If optimal configurations are rigid in this sense, we therefore expect many critical triangles.

\begin{boxedconjecture}\label{conj:critical}
An optimal configuration of $n\ge 5$ points has at least $2n-7$ critical triangles.
\end{boxedconjecture}

Figure~\ref{fig:critical-count-plot} tests Conjecture~\ref{conj:critical} against all configurations of Section~\ref{sec:solutions} and Appendix~\ref{app:best-known}, and Table~\ref{tab:active} lists the underlying counts.

\begin{figure}[ht]
  \centering
  \begin{tikzpicture}
    \begin{axis}[
      width=0.7\linewidth,
      height=5cm,
      xlabel={$n$},
      ylabel={\# critical triangles},
      xtick={5,6,...,16},
      ymin=0, ymax=68,
      xmin=4.5, xmax=16.5,
      grid=major,
      grid style={gray!30},
      mark size=2.5pt,
      every axis plot/.append style={thick},
      legend style={at={(0.03,0.97)}, anchor=north west, draw=none, fill=none, font=\small},
    ]
    \addplot[color=black, mark=*, mark options={fill=black}]
      coordinates {(5,4) (6,6) (7,8) (8,12) (9,11) (10,16)};
    \addlegendentry{Section~\ref{sec:solutions} ($n\le 10$)}
    \addplot[color=black, dashed, mark=o]
      coordinates {(10,16) (11,28) (12,20) (13,20) (14,26) (15,25) (16,64)};
    \addlegendentry{best known ($n\ge 11$)}
    \addplot[color=gray, thin, domain=5:16, samples=2] {2*x-7};
    \addlegendentry{$2n-7$}
    \end{axis}
  \end{tikzpicture}
  \caption{Number of critical triangles of the configurations of Section~\ref{sec:solutions} ($n\le 10$) and of the best-known configurations ($n\ge 11$), compared with the conjectured lower bound $2n-7$.}
  \label{fig:critical-count-plot}
\end{figure}

\begin{table}[ht]
  \centering
  \small
  \begin{tabular}{l *{12}{r}}
    \toprule
    $n$ & 5 & 6 & 7 & 8 & 9 & 10 & 11 & 12 & 13 & 14 & 15 & 16 \\
    \midrule
    critical triangles & 4 & 6 & 8 & 12 & 11 & 16 & 28 & 20 & 20 & 26 & 25 & 64 \\
    active box constraints & 6 & 6 & 7 & 8 & 8 & 8 & 8 & 8 & 7 & 8 & 6 & 8 \\
    active constraints $-(2n+1)$ & $-1$ & $-1$ & 0 & 3 & 0 & 3 & 13 & 3 & 0 & 5 & 0 & 39 \\
    \bottomrule
  \end{tabular}
  \caption{Active constraints of the configurations of Section~\ref{sec:solutions} ($n\le 10$) and of the best-known configurations ($n\ge 11$). For $n=6$, a member of the one-parameter family with $0<c<1/4$ is used.}
  \label{tab:active}
\end{table}

All configurations satisfy the conjectured bound, with equality for $n=9$. For $n=7$, $9$, $13$, and $15$, the number of active constraints equals the number of variables exactly, as the rigidity heuristic predicts. The outlier $n=16$, with $64$ critical triangles, is explained by the special structure of that configuration (Section~\ref{sec:noncritical-clustering}).
The heuristic is not a proof, since a nonlinear optimum can be isolated with fewer active constraints than variables. The configuration for $n=5$ is an example, with ten active constraints and eleven variables. It still satisfies the bound, as does the one-parameter family for $n=6$, because fewer than eight box constraints are active in both cases. For $n=6$, the missing active constraint matches the one degree of freedom of the family.

The conjecture is a lower bound on the number of active constraints. A remark of Oliver Stein points to an upper bound: if the linear independence constraint qualification (LICQ) holds, the gradients of the active constraints are linearly independent in $\mathbb{R}^{2n+1}$, so at most $2n+1$ constraints are active. At a rigid optimum at which LICQ holds, both bounds meet, and exactly $2n+1$ constraints are active. LICQ need not hold, however, and we checked it for all configurations of Table~\ref{tab:active}.

\begin{proposition}\label{prop:licq}
At the configurations of Section~\ref{sec:solutions}, LICQ holds for $n=5$, $7$, and $9$, and it fails for $n=6$, $8$, and $10$.
\end{proposition}

\begin{proof}
For $n=8$, twelve triangle constraints are active by the exact verification (Section~\ref{sec:solutions}), and the corners $(0,0)$ and $(1,1)$ and four further points on $\partial S$ give eight active box constraints. These $20$ gradients in $\mathbb{R}^{17}$ cannot be linearly independent. For $n=10$, sixteen triangle constraints and eight box constraints, one for each of the eight points on $\partial S$, give $24$ gradients in $\mathbb{R}^{21}$.
For the other values of $n$, we computed the rank of the matrix of active gradients exactly by Gaussian elimination over the respective number field. For $n=5$, $7$, and $9$, the rank equals the number of active constraints ($10$, $15$, and $19$). For the member $c=1/8$ of the family for $n=6$, the $12$ active gradients have rank $11$.
\end{proof}

The configurations for $n=7$ and $n=9$ are therefore nondegenerate in the strongest sense: exactly $2n+1$ constraints are active, and their gradients are linearly independent.
For $n=6$, LICQ fails although fewer constraints than variables are active, which reflects the one-parameter family of optima.
For the best-known configurations, LICQ fails whenever more than $2n+1$ constraints are active, that is, for $n=11$, $12$, $14$, and $16$. For $n=13$ and $n=15$, a numerical rank computation with 50 digits indicates that it holds, with smallest singular values $0.045$ and $0.022$ of the matrix of active gradients.
The script for these computations is part of the accompanying repository.

\subsection{Clustering of noncritical triangle areas}
\label{sec:noncritical-clustering}

For many configurations, including those of Section~\ref{sec:solutions}, the areas of the noncritical triangles do not spread out but cluster at a small number of values, and all areas in a cluster are equal.
Figure~\ref{fig:distinct-areas}a shows this for $n=5$: after the four critical triangles, the remaining six areas take only three values.
Figure~\ref{fig:distinct-areas}b shows the number of distinct noncritical areas relative to the number of noncritical triangles for $n=5,\dots,16$.

\begin{figure}[ht]
  \centering
  \begin{minipage}[c]{0.44\linewidth}\centering
    \includegraphics[width=\linewidth]{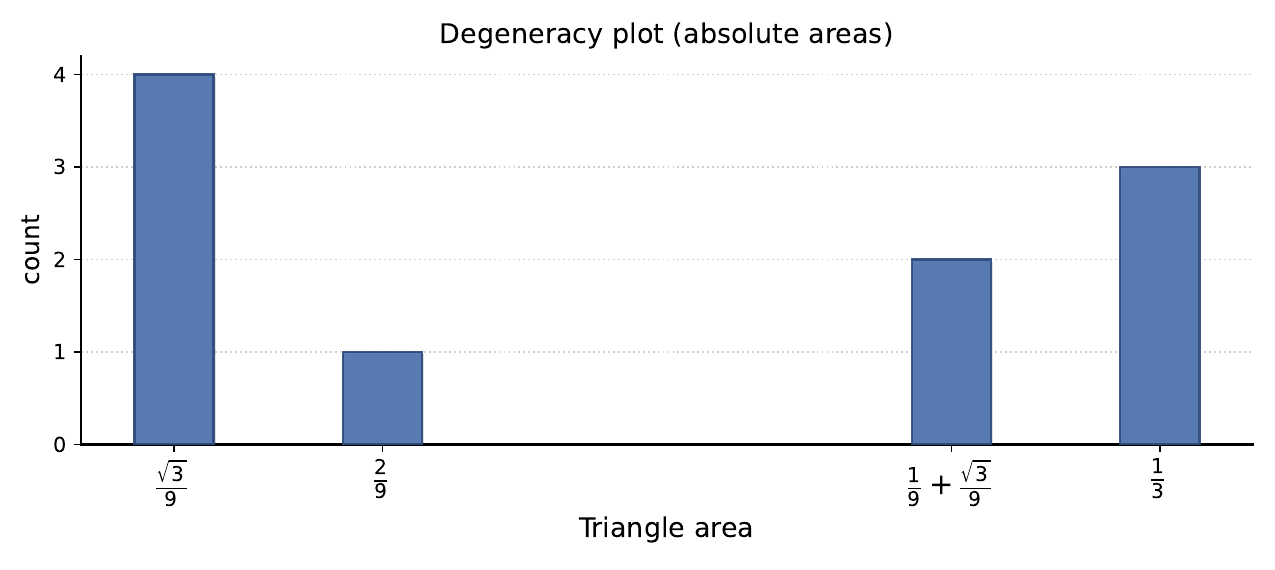}\\[2pt]
    {\small (a)}
  \end{minipage}\hfill
  \begin{minipage}[c]{0.52\linewidth}\centering
    \begin{tikzpicture}
    \begin{axis}[
      width=\linewidth,
      height=5.2cm,
      xlabel={$n$},
      ylabel={distinct / noncritical},
      xtick={5,6,...,16},
      ymin=0, ymax=1.05,
      xmin=4.5, xmax=16.5,
      grid=major,
      grid style={gray!30},
      mark size=2.5pt,
      legend style={at={(0.03,0.97)}, anchor=north west, draw=none, fill=white, font=\small},
    ]
    \addplot[only marks, mark=o, thick]
      coordinates {(5,0.500) (6,0.143) (8,0.182) (9,0.425) (10,0.212) (11,0.117) (12,0.080) (14,0.216) (16,0.060)};
    \addlegendentry{symmetric}
    \addplot[only marks, mark=*, mark options={fill=black}]
      coordinates {(7,0.481) (13,0.966) (15,0.984)};
    \addlegendentry{asymmetric}
    \end{axis}
  \end{tikzpicture}\\[2pt]
    {\small (b)}
  \end{minipage}
  \caption{(a)~Multiplicities of the distinct triangle areas of the optimal configuration for $n=5$: the minimum $\sqrt3/9$ is attained by four triangles, and the remaining six areas form three clusters of sizes $1$, $2$, and~$3$. (b)~Number of distinct noncritical triangle areas divided by the number of noncritical triangles, for the configurations of Section~\ref{sec:solutions} ($n\le 10$) and the best-known configurations ($n\ge 11$). Small values indicate strong clustering.}
  \label{fig:distinct-areas}
\end{figure}

Symmetry explains part of the clustering, since a symmetry of the configuration maps triangles to triangles of the same area, but not all of it.
The configuration for $n=7$ has no symmetry, yet its $27$ noncritical triangles take only $13$ distinct values, and one value is attained by five triangles, namely $(1,2,7)$, $(2,5,7)$, $(1,3,7)$, $(2,3,5)$, and $(4,5,7)$ in the labeling of Table~\ref{tab:n7-exact}.
For two of them there is a simple reason: $p_2$ and $p_7$ share their $x$-coordinate, so $(1,2,7)$ and $(2,5,7)$ have the common base $p_2p_7$ and their apex on the left edge. We have no comparably simple explanation for the other three.

The configuration for $n=16$ shows a different source of clustering. As observed by Tej Stead (personal communication, 2026), Beyleveld's configuration is an affine image of the vertex set of the four-dimensional cube: it consists of the points
\[
\left(\tfrac{2}{31},0\right)+\sum_{i=1}^4 \varepsilon_i e_i,\qquad \varepsilon\in\{0,1\}^4,
\]
with $e_1=\left(\tfrac{21}{31},0\right)$, $e_2=\left(-\tfrac{2}{31},\tfrac{10}{33}\right)$, $e_3=\left(\tfrac{8}{31},\tfrac{2}{33}\right)$, and $e_4=\left(0,\tfrac{7}{11}\right)$.
Every triangle area of such a point set is an integer combination of the six determinants $\det(e_i,e_j)$. In particular, all $560$ areas are integer multiples of $7/682$. They take only $31$ distinct values, and the minimum $7/341$ is attained by $64$ triangles.
This structure does not generalize in an obvious way: the optimal configuration for $n=8$ is not an affine image of the vertex set of the three-dimensional cube.

Not all configurations cluster, however. Almost all triangle areas of the asymmetric best-known configurations for $n=13$ and $n=15$ are distinct (Figure~\ref{fig:distinct-areas}b). Either these configurations are not optimal, or the clustering is limited to small $n$ and to configurations with special structure. This leads to the following question.

\begin{boxedquestion}\label{q:clustering}
Do the noncritical triangle areas of optimal Heilbronn configurations generally cluster at only a few values, and if so, why?
\end{boxedquestion}

\section{Outlook and final remarks}\label{sec:outlook}

We have presented an \emph{optimize-then-refine} framework for the Heilbronn triangle problem on the unit square. An MINLP model whose structure rests on Theorem~\ref{thm:case-split} solves all $n\le 10$ to global optimality within solver tolerances without precomputation, and exact symbolic computation turns the numerical solutions into exactly verified configurations.

Several directions remain open.
First, the upper bounds of this paper are numerical. A rigorous proof of optimality for $n=9$ would require the dual bounds to be validated, for example by interval-based global optimization~\cite{neumaier2004complete,cafieri2025covering}, by verified certificates for the branch-and-bound tree~\cite{cheung2017verifying}, or by an independent verification of the kind used for the traveling salesman problem~\cite{applegate2009certification}.
Second, the next cases are within reach. The step from $n=9$ to $n=10$ multiplied the running time by about 36, and a similar factor would put $n=11$ at about a week on a laptop. Since the growth comes from the dual bound, stronger relaxations or valid inequalities are the natural next step.
Third, Conjecture~\ref{conj:critical} on critical triangles and Open question~\ref{q:clustering} on the clustering of triangle areas remain open. The latter is directly relevant for the search for better configurations for $n=13$ and $n=15$.

More broadly, the two-step methodology applies to other packing and covering problems in discrete geometry whose conjectured optimal configurations involve algebraic numbers of moderate degree.

\section*{Acknowledgments}
We thank the two anonymous referees and the technical editor for their careful reading and their constructive comments, which improved the paper considerably.
We thank Tej Stead for maintaining the record tables of the Heilbronn problem and for pointing out the cube structure of the configuration for $n=16$,
Erich Friedman for kindly sharing the previously unpublished best-known configurations of Mark Beyleveld ($n=14,16$),
Rhys Chappell for the exact form of the configuration for $n=14$,
Amirali Modir for the connection between minimal covering parallelograms and the boundary structure,
and Oliver Stein for his remark on constraint qualifications.

\printbibliography

\appendix\label{appendix}

\section{Exact coordinates for $5\le n\le 10$}
\label{app:exact-coordinates}

The exact coordinates of the configurations of Section~\ref{sec:solutions} are listed in Tables~\ref{tab:exact-56}, \ref{tab:exact-89}, and~\ref{tab:exact-10}. For $n=7$, they are given in Table~\ref{tab:n7-exact}.

\begin{table}[ht]
  \centering
  \begin{tabular}[t]{c c c}
    \toprule
    Pt & $x$ & $y$ \\
    \midrule
    $1$ & $0$ & $\frac{1}{3}$ \\[3pt]
    $2$ & $\frac{\sqrt{3}}{3}$ & $0$ \\[3pt]
    $3$ & $1$ & $1 - \frac{\sqrt{3}}{3}$ \\[3pt]
    $4$ & $\frac{2}{3}$ & $1$ \\[3pt]
    $5$ & $0$ & $1$ \\
    \bottomrule
    \end{tabular}
  \hspace{1.5cm}
  \begin{tabular}[t]{c c c}
    \toprule
    Pt & $x$ & $y$ \\
    \midrule
    $1$ & $0$ & $c$ \\[3pt]
    $2$ & $\frac{1}{2}$ & $0$ \\[3pt]
    $3$ & $1$ & $\frac{1}{2} - c$ \\[3pt]
    $4$ & $\frac{1}{2}$ & $1$ \\[3pt]
    $5$ & $0$ & $c + \frac{1}{2}$ \\[3pt]
    $6$ & $1$ & $1-c$ \\
    \bottomrule
    \end{tabular}
  \caption{Exact coordinates for $n=5$ (left) and $n=6$ (right, with $c\in[0,\tfrac14]$).}
  \label{tab:exact-56}
\end{table}

\begin{table}[ht]
  \centering
  \small
  \begin{tabular}[t]{c c c}
    \toprule
    Pt & $x$ & $y$ \\
    \midrule
    $1$ & $0$ & $0$ \\[3pt]
    $2$ & $\frac{1}{6}+\frac{\sqrt{13}}{6}$ & $0$ \\[3pt]
    $3$ & $1$ & $\frac{7}{18}-\frac{\sqrt{13}}{18}$ \\[3pt]
    $4$ & $1$ & $1$ \\[3pt]
    $5$ & $0$ & $\frac{11}{18}+\frac{\sqrt{13}}{18}$ \\[3pt]
    $6$ & $\frac{5}{6}-\frac{\sqrt{13}}{6}$ & $1$ \\[3pt]
    $7$ & $\frac{5}{6}-\frac{\sqrt{13}}{6}$ & $\frac{7}{9}-\frac{\sqrt{13}}{9}$ \\[3pt]
    $8$ & $\frac{1}{6}+\frac{\sqrt{13}}{6}$ & $\frac{2}{9}+\frac{\sqrt{13}}{9}$ \\
    \bottomrule
    \end{tabular}
  \hspace{1cm}
  \begin{tabular}[t]{c c c}
    \toprule
    Pt & $x$ & $y$ \\
    \midrule
    $1$ & $0$ & $1-\frac{\sqrt{65}}{10}$ \\[3pt]
    $2$ & $\frac{3}{8}-\frac{\sqrt{65}}{40}$ & $0$ \\[3pt]
    $3$ & $1$ & $\frac{9}{16}-\frac{3\sqrt{65}}{80}$ \\[3pt]
    $4$ & $\frac{3}{8}-\frac{\sqrt{65}}{40}$ & $1$ \\[3pt]
    $5$ & $0$ & $\frac{5}{8}+\frac{\sqrt{65}}{40}$ \\[3pt]
    $6$ & $\frac{1}{4}+\frac{\sqrt{65}}{20}$ & $\frac{3}{4}-\frac{\sqrt{65}}{20}$ \\[3pt]
    $7$ & $\frac{7}{16}+\frac{3\sqrt{65}}{80}$ & $0$ \\[3pt]
    $8$ & $\frac{\sqrt{65}}{10}$ & $1$ \\[3pt]
    $9$ & $1$ & $\frac{5}{8}+\frac{\sqrt{65}}{40}$ \\
    \bottomrule
    \end{tabular}
  \caption{Exact coordinates for $n=8$ (left) and $n=9$ (right).}
  \label{tab:exact-89}
\end{table}

\begin{table}[ht]
  \centering
  \small
  \begin{tabular}[t]{c c c}
    \toprule
    Pt & $x$ & $y$ \\
    \midrule
    $1$ & $0$ & $1-3g+2g^2$ \\[3pt]
    $2$ & $1-3g+2g^2$ & $0$ \\[3pt]
    $3$ & $1$ & $3g-2g^2$ \\[3pt]
    $4$ & $3g-2g^2$ & $1$ \\[3pt]
    $5$ & $0$ & $1-\frac{g}{2}$ \\[3pt]
    $6$ & $\frac{g}{2}$ & $1$ \\[3pt]
    $7$ & $g$ & $g$ \\[3pt]
    $8$ & $1-g$ & $1-g$ \\[3pt]
    $9$ & $1-\frac{g}{2}$ & $0$ \\[3pt]
    $10$ & $1$ & $\frac{g}{2}$ \\
    \bottomrule
  \end{tabular}
  \caption{Exact coordinates for $n=10$, where $g\approx 0.3156$ is the real root of $12g^3-27g^2+20g-4$, explicitly $g=\frac34-\frac{(63+8\sqrt{62})^{1/3}}{12}-\frac{1}{12(63+8\sqrt{62})^{1/3}}$. The minimum area is $\Delta(P)=\frac58 g^2-\frac12 g^3$.}
  \label{tab:exact-10}
\end{table}

\FloatBarrier
\section{Best-known configurations for $11\le n\le 16$}
\label{app:best-known}

For completeness, we collect the best-known configurations for $11\le n\le 16$ as of October 2026. Records for larger $n$ are maintained on the website of Tej Stead~\cite{SteadHeilbronn}, and coordinates with 30 significant digits for all configurations below are available in the accompanying repository.
The configuration for $n=12$ is due to Comellas and Yebra~\cite{comellasyebra2002},
the one for $n=11$ to Goldberg~\cite{goldberg1972maximizing}, the one for $n=13$ to Peter Karpov~\cite{KarpovAscension}, the ones for $n=14$ and $n=16$ to Mark Beyleveld, and the one for $n=15$ to the author (found in 2026 with a heuristic search).
For $n=13$ we list Karpov's configuration with fully converged coordinates as given in~\cite{SteadHeilbronn}. The stated lower bounds refer to the full-precision coordinates. For $n=14$, Rhys Chappell~\cite{ChappellN14} found an exact form: $\Delta_{14}$ is bounded below by the smallest positive root of $320x^3+768x^2-60x+1$.
The exact coordinates are listed in
Tables~\ref{tab:bestknown}, \ref{tab:bestknown-1314}, and~\ref{tab:bestknown-1516}.
In the entry for $n=12$, let $x = 1 - \frac{(27+3\sqrt{57})^{2/3}+6}{6(27+3\sqrt{57})^{1/3}} \approx 0.1154$ and $y = 2x^2-3x+\frac{1}{2}$.

\begin{table}[ht]
\centering
\caption{Best-known configurations for $n=11$, due to Goldberg~\cite{goldberg1972maximizing}, and $n=12$, due to Comellas and Yebra~\cite{comellasyebra2002}.}
\label{tab:bestknown}
{\small
\begin{tabular}{r c@{\hskip 6pt} c @{\hskip 14pt} r r@{\hskip 6pt} r}
\toprule
\multicolumn{3}{c}{$n=11$} & \multicolumn{3}{c}{$n=12$} \\
\# & $x$ & $y$ & \# & $x$ & $y$ \\
\midrule
$0$ & $1/3$ & $0$     & $0$ & $x$ & $0$ \\
$1$ & $2/3$ & $0$     & $1$ & $1{-}x$ & $0$ \\
$2$ & $0$ & $2/9$     & $2$ & $0$ & $x$ \\
$3$ & $1$ & $2/9$     & $3$ & $1$ & $x$ \\
$4$ & $1/3$ & $4/9$   & $4$ & $1/2$ & $y$ \\
$5$ & $2/3$ & $4/9$   & $5$ & $y$ & $1/2$ \\
$6$ & $0$ & $2/3$     & $6$ & $1{-}y$ & $1/2$ \\
$7$ & $1$ & $2/3$     & $7$ & $1/2$ & $1{-}y$ \\
$8$ & $1/2$ & $7/9$   & $8$ & $0$ & $1{-}x$ \\
$9$ & $1/6$ & $1$     & $9$ & $1$ & $1{-}x$ \\
$10$ & $5/6$ & $1$    & $10$ & $x$ & $1$ \\
&       &        & $11$ & $1{-}x$ & $1$ \\
\midrule
\multicolumn{3}{c}{$\Delta_{11} \ge 1/27 \approx 0.03704$}
& \multicolumn{3}{c}{$\Delta_{12} \ge \tfrac{1}{4}x {+} \tfrac{1}{2}xy {-} \tfrac{1}{2}x^2 \approx 0.03260$} \\
\bottomrule
\end{tabular}}
\end{table}

\begin{table}[ht]
\centering
\caption{Best-known configurations for $n=13$ and $n=14$ (coordinates rounded to ten decimal places).}
\label{tab:bestknown-1314}
{\small
\begin{tabular}{r r@{\hskip 6pt} r @{\hskip 14pt} r r@{\hskip 6pt} r}
\toprule
\multicolumn{3}{c}{$n=13$} & \multicolumn{3}{c}{$n=14$} \\
\# & $x$ & $y$ & \# & $x$ & $y$ \\
\midrule
0 & 0.9647677640 & 0.0876311313 & 0 & 0.0776202061 & 0 \\
1 & 0 & 1 & 1 & 0.9223797939 & 1 \\
2 & 0.8969922641 & 0.9025311902 & 2 & 0.9223797939 & 0 \\
3 & 0.7613614847 & 0.4419841319 & 3 & 0.0776202061 & 1 \\
4 & 0.6552877381 & 1 & 4 & 0 & 0.1868859173 \\
5 & 0.7485358270 & 0 & 5 & 1 & 0.8131140827 \\
6 & 0 & 0.0992617012 & 6 & 1 & 0.1868859173 \\
7 & 1 & 0.4613514603 & 7 & 0 & 0.8131140827 \\
8 & 0.3285229862 & 0.6333065927 & 8 & 0.2923326509 & 0.3213448977 \\
9 & 0.0879505163 & 0.6144345209 & 9 & 0.7076673491 & 0.6786551023 \\
10 & 0.3450773648 & 0.9015161285 & 10 & 0.7076673491 & 0.3213448977 \\
11 & 0.0879505163 & 0 & 11 & 0.2923326509 & 0.6786551023 \\
12 & 0.5001295156 & 0.1492457946 & 12 & 0.5000000000 & 0.1382779581 \\
 & &  & 13 & 0.5000000000 & 0.8617220419 \\
\midrule
\multicolumn{3}{c}{$\Delta_{13} \ge 0.0270199166$}
& \multicolumn{3}{c}{$\Delta_{14} \ge 0.0243039796$} \\
\bottomrule
\end{tabular}}
\end{table}

\begin{table}[ht]
\centering
\caption{Best-known configurations for $n=15$ (coordinates rounded to ten decimal places) and $n=16$. All coordinates for $n=16$ are rational.}
\label{tab:bestknown-1516}
{\small
\begin{tabular}{r r@{\hskip 6pt} r @{\hskip 14pt} r r@{\hskip 6pt} r}
\toprule
\multicolumn{3}{c}{$n=15$} & \multicolumn{3}{c}{$n=16$} \\
\# & $x$ & $y$ & \# & $x$ & $y$ \\
\midrule
0 & 1 & 0.3700738107 & 0 & $2/31$ & $0$ \\
1 & 0.6740789364 & 0 & 1 & $29/31$ & $1$ \\
2 & 0.9518729818 & 0.0568746017 & 2 & $23/31$ & $0$ \\
3 & 0.9609111977 & 1 & 3 & $8/31$ & $1$ \\
4 & 0.2757403554 & 1 & 4 & $0$ & $10/33$ \\
5 & 0 & 0.0151135780 & 5 & $1$ & $23/33$ \\
6 & 0 & 0.9377422462 & 6 & $1$ & $2/33$ \\
7 & 0.6336710221 & 0.6204482905 & 7 & $0$ & $31/33$ \\
8 & 0.0459784998 & 0.3331833868 & 8 & $8/31$ & $4/11$ \\
9 & 0.7219449621 & 0.3148563442 & 9 & $23/31$ & $7/11$ \\
10 & 0.3239949629 & 0.3915651931 & 10 & $10/31$ & $2/33$ \\
11 & 0.2308100985 & 0.6795836111 & 11 & $21/31$ & $31/33$ \\
12 & 0.2796220233 & 0.0719471944 & 12 & $21/31$ & $10/33$ \\
13 & 0.6813782935 & 0.9377422462 & 13 & $10/31$ & $23/33$ \\
14 & 0.9129108987 & 0.6847586787 & 14 & $29/31$ & $4/11$ \\
 & &  & 15 & $2/31$ & $7/11$ \\
\midrule
\multicolumn{3}{c}{$\Delta_{15} \ge 0.0212105410$}
& \multicolumn{3}{c}{$\Delta_{16} \ge 7/341 \approx 0.02053$} \\
\bottomrule
\end{tabular}}
\end{table}

\end{document}